\documentclass[10pt]{amsart}

\usepackage[colorlinks]{hyperref}
\usepackage{color,graphicx,shortvrb}
\usepackage[latin 1]{inputenc}

\usepackage[active]{srcltx} 

\usepackage{enumerate}
\usepackage{amssymb}
\usepackage{tikz-cd}

\newtheorem{theorem}{Theorem}[section]

\newtheorem{corollary}[theorem]{Corollary}

\newtheorem{definition}[theorem]{Definition}
\newtheorem{lemma}[theorem]{Lemma}

\newtheorem{proposition}[theorem]{Proposition}
\newtheorem{remark}[theorem]{Remark}

\newtheorem{example}[theorem]{Example}

\def\11{\textbf{$1$}}

\DeclareGraphicsExtensions{.jpg,.pdf,.png,.eps}

\begin{document}

\numberwithin{equation}{section}

\title[The minimax principle and related topics in the Jordan setting]{The minimax principle and related topics in the Jordan setting}

\author[Fern\'{a}ndez-Polo]{Francisco J. Fern\'{a}ndez-Polo}
\email{pacopolo@ugr.es}
\address{Departamento de An{\'a}lisis Matem{\'a}tico, Facultad de
Ciencias, Universidad de Granada, 18071 Granada, Spain.}

\thanks{Authors partially supported by the Spanish Ministry of Science,  Innovation and Universities
project no. PGC2018-09332-B-I00}

\subjclass[2010]{Primary 46L70;  15A42 ; 17C65} 

\keywords{weakly compact JB$^*$-triples, minimax principle, Weyl inequalities, Ky Fan maximum principle, perturbation of eigenvectors}

\date{}

\begin{abstract} We prove a minimax principle for weakly compact JB$^*$-triples characterizing geometrically the singular values of an element. Among the consequences of this principle we present a Weyl inequality on the perturbation of the singular values and a Cauchy-Poincaré  (interlacing) theorem. We also obtain a version of the Ky Fan maximum principle in the setting of weakly compact JB$^*$-triples. We study perturbations of the spectral resolutions showing that small perturbations of an element produces small perturbations of the corresponding spectral resolutions. As a consequence we obtain that weakly compact JB$^*$-triples satisfy the property that perturbations of a convex combination of elements in the closed unit ball coincide with a convex combination of perturbations of the elements also in the closed unit ball. All these results hold true when particularized to weakly compact  JB$^*$-algebras.
\end{abstract}

\maketitle
\thispagestyle{empty}

\section{Introduction}\label{sec:intro}



The celebrated minimax principle provides a characterization of the eigenvalues of a symmetric compact operator on a Hilbert space without any reference to eigenvectors or the characteristic polynomial. It is also known as Courant-Fisher Minimax theorem after the contributions of E. Fischer in the finite dimensional case \cite{Fis} and R. Courant in the case of hermitian compact operators on an infinite dimensional Hilbert space \cite{Cou}. Concretely, given a Hilbert space $H$ and a symmetric compact operator $A$ acting on $H$ with eigenvalues $(\lambda_n(A))_{n\in \mathbb{N}}$ (arranged in decreasing order and counting multiplicity) then for every natural $n$ we have that\begin{align} \lambda_n(A) &=\max_{M} \min_{\varphi\in M, \|\varphi\|=1} <A\varphi, \varphi> \notag  \\  & = \min_{N} \max_{\varphi\in N, \|\varphi\|=1} <A\varphi, \varphi>, \notag\end{align} where $M$ runs over the subspaces of $H$ of dimension $n$ and $N$ runs over the subspaces of $H$ of codimension $n-1$.\smallskip

Among the multiple consequences of the minimax principle we highlight Cauchy's Interlacing Theorem and Weyl's inequalities. The first result gives interesting inequalities between the eigenvalues of a Hermitian matrix A and those of a principal submatrix. H. Weyl showed the continuity of the eigenvalues in \cite{Wey}, initiating a study of inequalities for eigenvalues and singular values which was continued, among others, by G. Polya, A. Horn and K. Fan. In particular the Ky Fan maximum principle \cite{Fan49} provides an extremum property of the sum of the first $k$ eigenvalues of a symmetric compact operator on a Hilbert space.\bigskip

The minimax principle can be stated for a non-necessarily symmetric compact operator, $A$, replacing the eigenvalues of $A$ with its singular values (the eigenvalues of $(AA^*)^{\frac 12}$) and the inner products $<A\varphi, \varphi>$ with $\|A\varphi\|$.\medskip

Although Weyl's inequality shows that there is a dependence between the distance of the corresponding eigenvalues of two compact operators this is no longer true when dealing with the corresponding eigenvectors. However, C. Davis proved that there is such a relation between the corresponding spectral resolutions of the operators whenever they are ``close enough'' (see \cite{Dav}). As observed by J. Becerra and the author of this note in \cite[Theorem 3.6]{BecFer}, it can derived from Davis' results the continuity of the spectral resolutions in case of non-necessary symmetric elements in finite dimensional C$^*$-algebras.\bigskip

The main goal of this work is to extend all these results to the Jordan setting, where by Jordan setting we mean the category of JB$^*$-algebras (the complexifications of  JB-algebras) and their ternary generalization JB$^*$-triples. We recommend \cite{HanStor} as a reference book on JB-algebras while in section \ref{sec: preliminaries} we will survey some basics on JB$^*$-triples.\smallskip

In this non-associative setting we only have been able to find results in case of Euclidean Jordan algebras, which happen to be finite dimensional JB-algebras by \cite[Page 42]{FarKor} and \cite[3.1.7]{HanStor}. The first stunning result in this area is the minimax principle obtained by U. Hirzebruch in 1970 \cite{Hir}. The implications of this result seemed to be unnoticed for almost forty years until  M.S. Gowda, J. Tao and  M. Moldovan used it to derive a Weyl's perturbation inequality in Euclidean Jordan algebras (see \cite[Theorem 9]{GowTaoMol}). Gowda and Tao also obtained a Cauchy-Poincaré interlacing theorem  in \cite{GowTao}. Unfortunately, unlike in the case finite dimensional C$^*$-algebras, these results can not be used directly to derive their analogs for non-necessarily symmetric elements in a finite dimensional JB$^*$-algebra. \bigskip

In Section \ref{sec: minmax} we obtain a generalized minimax principle for weakly compact JB$^*$-triples (see Theorem \ref{t minmax}). In this geometric characterization of the singular values of an element, arranged in decreasing order and counting multiplicity, the role played by the subspaces and the elements of the Hilbert space is now played by tripotents and minimal tripotents respectively. As consequences of this result a Weyl inequality and a Cauchy-Poincaré  interlacing Theorem are also obtained.\bigskip

Section \ref{sec: Ky Fan} is entirely devoted to present a generalized Ky Fan maximum principle in weakly compact JB$^*$-triples. More concretely,  $\ell_p$-norms ($p\geq 1$) of the first $n$ singular values of an element are geometrically characterized in Theorem \ref{t Ky Fan maximum principle}.\bigskip

The problem of the continuity of the spectral resolutions is treated in Section \ref{sec: perturbation of eigenvectors}. If we consider the singular values of an element arranged in decreasing order but not counting multiplicity, for every singular value we get a finite-rank tripotent (the spectral resolution associated to a point). By Weyl inequality, given a (conveniently small) perturbation of the element, for each of this finite-rank tripotents we can associate another tripotent of the same rank being an spectral resolution of the perturbation. In Theorem \ref{t continuity of spectral resolutions} it is shown the continuity of the associated spectral resolutions. This result provides a generalization of the already mentioned results of Davis and Becerra and Fernández-Polo to the setting of weakly compact JB$^*$-triples.\bigskip

In the last section of this manuscript we will show that every weakly compact JB$^*$-triple satisfy the  property $\hbox{\rm{(co)}}$ (see definition \ref{d (co)}), that is, roughly speaking, perturbations of a convex combination of elements in its closed unit ball coincide with a convex combination of perturbations of the elements also in its closed unit ball. This result generalizes \cite[Theorem 3.8]{BecFer} where it was proved that finite dimensional C$^*$-algebras have the property $\hbox{\rm{(co)}}$.\medskip

\section{preliminaries}\label{sec: preliminaries}

Given a Banach space $X$ we will denote by $\mathcal{B}_X$ its closed unit ball.\smallskip

We recall that a \emph{JB$^*$-triple} is a complex Banach space $\mathcal{U}$ which can be
equipped with a continuous triple product $\{\cdot,\cdot,\cdot \}:
\mathcal{U}\times \mathcal{U}\times \mathcal{U} \to \mathcal{U},$ which is symmetric and linear in the
first and third variables, conjugate linear in the second variable and satisfies the following axioms
\begin{enumerate}[{\rm (a)}] \item $L(a,b) L(x,y) = L(x,y) L(a,b) + L(L(a,b)x,y)
 - L(x,L(b,a)y),$
where $L(a,b)$ is the operator on $\mathcal{U}$ given by$L(a,b) x = \{a,b,x\};$
\item $L(a,a)$ is an hermitian operator with non-negative
spectrum; \item $\|L(a,a)\| = \|a\|^2$.\end{enumerate}

Examples of JB$^*$-triples include C$^*$-algebras with respect to the triple product defined by
 \begin{equation}\label{eq product operators} \{x,y,z\} =\frac12 (x y^* z +z y^*x),
\end{equation} and JB$^*$-algebras (in the sense of \cite{Wri77}) under the triple product $\{x,y,z\}= (x\circ y^*) \circ z + (z\circ y^*)\circ x - (x\circ z)\circ y^*.$ The so-called \emph{ternary rings of operators} (TRO's) studied, for example, in \cite{NeaRus} are also examples of JB$^*$-triples.\smallskip

Suppose $x$ is an element in a JB$^*$-triple $\mathcal{U}$. The symbol $\mathcal{U}_x$ will denote the JB$^*$-subtriple generated by $x$, that is, the closed subspace generated by all odd powers of the form $x^{[1]} := x$, $x^{[3]} := \{x,x,x\}$, and $x^{[2n+1]} := \{x,x,x^{[2n-1]}\},$ $(n\in \mathbb{N})$. It is known that $\mathcal{U}_x$ is JB$^*$-triple isomorphic (and hence isometric) to a commutative C$^*$-algebra in which $x$ is a positive generator (cf. \cite[Corollary 1.15]{Kau83}). We identify the triple spectrum of $x$ with the spectrum of (the $C^*$-algebra) $\mathcal{U}_x$.\medskip

An element $e$ in a JB$^*$-triple $\mathcal{U}$ is said to be a \emph{tripotent} if $\{e,e,e\} =e$. For each tripotent $e$ in $\mathcal{U}$ the eigenvalues of the operator $L(e,e)$ are contained in the set $\{0,1/2,1\},$ and $\mathcal{U}$ can be decomposed in the form  $$\mathcal{U}= \mathcal{U}_{2} (e) \oplus \mathcal{U}_{1} (e) \oplus \mathcal{U}_0 (e),$$ where for
$i=0,1,2,$ $\mathcal{U}_i (e)$ is the $\frac{i}{2}$ eigenspace of $L(e,e)$. This decomposition is known as the \emph{Peirce decomposition}\label{eq Peirce decomposition} associated with $e$. The so-called \emph{Peirce arithmetic} (also called Peirce rules) affirms that $\{\mathcal{U}_{i}(e),\mathcal{U}_{j} (e),\mathcal{U}_{k} (e)\}\subseteq \mathcal{U}_{i-j+k} (e)$ if $i-j+k \in \{ 0,1,2\},$ and $\{\mathcal{U}_{i}(e),\mathcal{U}_{j} (e),\mathcal{U}_{k} (e)\}=\{0\}$ otherwise, and $$\{\mathcal{U}_{2} (e),\mathcal{U}_{0}(e),\mathcal{U}\} =\{\mathcal{U}_{0} (e),\mathcal{U}_{2}(e),\mathcal{U}\} =0.$$
The projection $P_{k_{}}(e)$ of $\mathcal{U}$ onto $\mathcal{U}_{k} (e)$ is called the Peirce $k$-projection. It is known that Peirce projections are contractive (cf. \cite{FriRus85}) and satisfy that $P_{2}(e) = Q(e)^2,$ $P_{1}(e) =2(L(e,e)-Q(e)^2),$ and $P_{0}(e) =Id_E - 2 L(e,e) + Q(e)^2,$ where $Q(e):\mathcal{U}\to \mathcal{U}$ is the conjugate linear map given by $Q(e) (x) =\{e,x,e\}$. The Peirce subspace $\mathcal{U}_2(e)$ is a JB$^*$-algebra with Jordan product $x\circ y=\{x,e,y\}$ and involution $Q(e)$.\smallskip

There is a partial order in the set of tripotents given by $e\leq f$ if and only if $P_2(e)f=e$ (equivalently, $f=e+P_0(e)f$ or $e$ is a projection in the JB$^*$-algebra $\mathcal{U}_2(f)$).\medskip

Two elements $x,y$ in $\mathcal{U}$ are orthogonal ($x\perp y$) whenever $\{x,x,y\}=0$ (see \cite{BurFerGarMarPer} for some equivalent definitions). When particularized to tripotents $e,f$ we have that $e\perp f $ if and only if $P_0(e)f=f$. A tripotent $e$ in $\mathcal{U}$ is said to be \emph{minimal} if $\mathcal{U}_2(e)=\mathbb{C} e \neq \{0\}$ and \emph{finite rank} if $e$ is the finite sum of mutually orthogonal minimal tripotents in $\mathcal{U}$. Analogously, an element in $\mathcal{U}$ is said to be of finite rank if it is a finite linear combination of mutually orthogonal minimal tripotents in $\mathcal{U}$. The \emph{rank} of a finite rank element $x$, $rank(x)$, is the minimum number of mutually orthogonal minimal tripotents which can be used to express $x$. We shall consider that the rank of the zero tripotent is 0.\medskip

A JBW$^*$-triple is a JB$^*$-triple which is also a dual Banach space (with a unique isometric predual \cite{BarTim86}). It is known that the second dual of a JB$^*$-triple is a JBW$^*$-triple (compare \cite{Din86}). An extension of Sakai's theorem assures that the triple product of every JBW$^*$-triple is separately weak$^*$-continuous (cf. \cite{BarTim86} or \cite{Horn87}).\medskip

Another illustrative examples of JBW$^*$-triples are given by the so-called Cartan factors, whose classification generalize that for finite dimensional JB$^*$-algebra factors given by P. Jordan, J. von Neumann and E. Wigner in \cite{JorNeuWig34}. We will present them in order to make the notion of weakly compact JB$^*$-triple more approachable.\medskip

A complex Banach space is a \emph{Cartan factor of type 1} is it coincides with the complex Banach space $L(H, K)$, of all bounded linear operators between two complex Hilbert spaces, $H$ and $K$, whose triple product is given by \eqref{eq product operators}.

Given a conjugation, $j$, on a complex Hilbert space, $H$, we can define a linear involution on $L(H)$ defined by $x \mapsto x^{t}:=j x^* j$. A \emph{type 2 Cartan factor} is a subtriple of $L(H)$ formed by the skew-symmetric operators for the involution $t$; similarly, a \emph{type 3 Cartan factor} is formed by the $t$-symmetric operators.  A Banach space $X$ is called a \emph{Cartan factor of type 4} or \emph{spin} if $X$ admits a complete inner product $(.|.)$ and a conjugation $x\mapsto \overline{x},$ for which the norm of $X$ is given by $$ \|x\|^2 = (x|x) + \sqrt{(x|x)^2 -|
(x|\overline{x}) |^2}.$$ \emph{Cartan factors of types 5 and 6} (also called \emph{exceptional} Cartan factors) are both finite dimensional and coincide with the $1\times 2$-matrices and symmetric $3\times 3$-matrices over the complex Cayley numbers, respectively.\smallskip

We recall now some terminology taken from \cite{BunChu92}. Let $K(H,H')$ be the space of all compact linear operators between two complex Hilbert spaces. We shall write $K(H)$ instead of $K(H,H)$. If $C_j$ is a Cartan factor of type $j\in \{1,2,3,4,5,6\}$, we define $K_1 = K(H,H')$ for $C_1= L(H,H')$, $K_j=C_j\cap K(H)$ for $j=2,3$, and in the remaining cases $K_4 = C_4,$ $K_5 = C_5$, and $K_6 = C_6$. The JB$^*$-triples $K_1,K_2,\ldots, K_6$ are called elementary JB$^*$-triples.  The class of \emph{weakly compact JB$^*$-triples} consist of those JB$^*$-triples which are a (possibly infinite) $c_0$-sum of elementary JB$^*$-triples (see \cite[Theorem 2.4]{BunChu92}).\medskip

Given a JBW*-triple $W,$ a norm-one element $\varphi$ of $W_*$ and a norm-one element $z$ in $W$ such that $\varphi (z) =1$, it follows from \cite[Proposition 1.2]{BarFri} that the assignment
$$(x,y)\mapsto \varphi\{x,y,z\}$$ defines a positive sesquilinear form on $W,$ the values of which are independent of choice of $z$, and induces a prehilbert seminorm on $W$ given by $$ \|x\|_{\varphi}:= \left(\varphi\ \{x,x,z\}\right)^{\frac{1}{2}}.$$ As $\varphi$ ranges over the unit sphere of $W_*$ the topology induced by these seminorms is termed the strong*-topology of $W$. Moreover, $\| \cdot \|_{\varphi}$ is additive on the sum of orthogonal elements (see \cite[Lemma 3.3]{FerPer06} for a more general result).\smallskip

Whenever $\mathcal{U}$ is a weakly compact JB$^*$-triple, associated to every minimal tripotent $v$ in $\mathcal{U}$ there exists, $\varphi_v$, an (unique) extreme point in the closed unit ball of $\mathcal{U}^*$ satisfying $\varphi_v=\varphi_v\circ P_2(v)$ and $\varphi_v(v)=1$ (cf. \cite[Proposition 4]{FriRus85}). Therefore we can define the seminorm $\| \cdot\|_v=\|\cdot\|_{\varphi_v}$ in $\mathcal{U}$.  Since $v$ is minimal and $P_2(v)\{x,x,v\}$ is positive in $\mathcal{U}_(v)=\mathbb{C}v$, we also have that $$\| x\|_v^2=\|x\|_{\varphi_v}^2= \varphi_v\circ P_2(v)\{x,x,v\}=\|P_2(v)\{x,x,v\}\|.$$ This particular seminorms will be widely used throughout the present work and where introduced with more generality in \cite{BunFerMarPer}.

\bigskip
Let $u,v$ be tripotents in a JB$^*$-triple $\mathcal{U}$. We say that $u$ and $v$ are \emph{collinear} (written $u\top v$) if $u\in \mathcal{U}_1(v)$ and $v\in \mathcal{U}_1(u)$.  We say that $u$ \emph{governs} $v$, $u \vdash v$, whenever $v\in \mathcal{U}_{2} (u)$ and $u\in \mathcal{U}_{1} (v)$.\smallskip

We will refer to the following statement as the \emph{extreme ray property} for JB$^*$-triples and combines some well-known results by Y. Friedman and B. Russo (see \cite[Proposition 5 and Proposition 6]{FriRus85} and \cite[Proposition 1.5]{FriRus93}) with the generation of Cartan factors by orthonormal grids presented in \cite{DanFri87}.

\begin{proposition}\label{p generalised extreme ray property}
Let $\mathcal{U}$ be a JB$^*$-triple. Given a minimal tripotent $v$ in $\mathcal{U}$, then for every arbitrary tripotent  $e$ in $\mathcal{U}$,  $P_2(e)v$ and $P_0(e)v$ are a multiple of some minimal tripotent in $\mathcal{U}$. More concretely, one of the following happens
\begin{enumerate}
\item There exists $ \alpha, \beta, \gamma, \delta \in \mathbb{C}$ with $|\alpha|^2+|\beta|^2 + |\gamma|^2+ |\delta|^2=1$, $\alpha \delta  = \beta \gamma$ and minimal tripotents $v_{11},v_{12},v_{21},v_{22}$ (zero if the corresponding coefficient vanishes) satisfying $v_{11}\perp v_{22}$, $v_{12}\perp v_{21}$, $v_{11}\top v_{12}\top v_{22}\top v_{21}$,  $v_{11}\in \mathcal{U}_2(e)$, $v_{21},v_{12} \in \mathcal{U}_1(e)$ and $v_{22}\in \mathcal{U}_0(e)$, such that $$v = \alpha v_{11}+\beta v_{12}+\gamma v_{21}+\delta v_{22}$$
\item There exists $ \alpha, \beta, \delta \in \mathbb{C}$ with $|\alpha|^2+ 2|\beta|^2 + |\delta|^2=1$, $\alpha \delta  = \beta^2$, minimal tripotents $v_{11}\in \mathcal{U}_2(e)$, $v_{22}\in \mathcal{U}_0(e)$ and a tripotent $w\in \mathcal{U}_1(e)$ (each of these tripotents zero if the corresponding coefficient vanishes), satisfying $w \vdash v_{11}$, $w \vdash v_{22}$ and $v_{11}\perp v_{22}$ such that $$ v = \alpha v_{11}+\beta  w +\delta v_{22}.$$

\end{enumerate}

\end{proposition}

Notice that in both cases described above, $|\alpha|=\|P_2(e)v\|$, $|\delta| =\|P_0(e)v\|$ and $(|\alpha|+|\delta|)^2\leq 1$.\smallskip

\begin{lemma}\label{l extreme ray prop en Peirce 1}
Let $\mathcal{U}$ be a JB$^*$-triple and let $e$ be a tripotent in $\mathcal{U}$. Then
\begin{itemize}

\item[a)] For every minimal tripotent $v$ in $\mathcal{U}_1(e)$ we have that $u=2\{v,v,e\}$ is a minimal projection in $\mathcal{U}_2(e)$ and $2\{v,v,e\}=\{u,u,e\}$.

\item[b)] For every rank-2 tripotent $w$ in $\mathcal{U}_1(e)$ with $\mathcal{U}_2(e)\cap \mathcal{U}_2(w)\neq \{0\} $ there exists a minimal tripotent $u \in \mathcal{U}_2(e)$ such that $\{w,w,e\}=\{u,u,e\}$.
\end{itemize}
\end{lemma}

\begin{proof}
a) Given a minimal tripotent $v$ in $\mathcal{U}_1(e)$ we have that $v,e$ are compatible, i.e. their Peirce projections commute (see \cite[Lemma 1.10]{FriRus85}), thus Peirce arithmetics give  $P_2(v)e=0$  and  $\{v,v,e\}$ is an element in $\mathcal{U}_2(e)\cap \mathcal{U}_1(v)$ positive in $\mathcal{U}_2(e)$ and not zero by \cite[Lemma 1.5]{FriRus85}. We denote by  $u$ the range projection of $\{v,v,e\}$ in $\mathcal{U}_2^{**}(e)$ so $u\in \mathcal{U}_1^{**}(v)$.  We recall that $\mathcal{U}^{**}_1(v)$ is a JBW$^*$-triple of rank smaller than or equal to 2 \cite[Corollary 2.2]{DanFri87}. If we assume that $u$ has rank two, Peirce arithmetics show that $\{u,v,u\}=0$ while \cite[Proposition 2.1]{DanFri87} gives $\|\{u,v,u\}\|=1$ which leads to a contradiction.
Therefore $u$ is a minimal projection in $\mathcal{U}_2(e)\cap \mathcal{U}_1(v)$, $u$ and $v$ are collinear by \cite[Proposition 2.1]{DanFri87} and $\{v,v,e\}$  is a multiple of $u$. Using this information and Peirce arithmetics we get that  $\|\{v,v,e\}\|u=\{v,v,e\}=\{v,v,u\}+\{v,v,e-u\}=\frac 12 u+\{v,v,e-u\} $ and hence $\{v,v,e\}=\frac 12 u= \frac 12 \{u,u,e\}$.\medskip

b) By our assumptions, $\mathcal{U}_2(w)$ is a JB$^*$-algebra of rank 2 with involution $Q(w)$ (a conjugate linear automorphism of order 2). Let $x$ be a norm-one element in $\mathcal{U}_2(e)\cap \mathcal{U}_2(w)$. Clearly $Q(w)x$ is also a norm-one element orthogonal to $x$ by Peirce arithmetics. This implies that the rank of both $x$ and $Q(w)x$ has to be one and hence $x$ is a minimal tripotent in $\mathcal{U}_2(w)$ and also in $\mathcal{U}$. We define $u=x$.\smallskip

By \cite[Proposition 2.1]{DanFri87} we have that $(u,w,Q(w)u)$ form a trangle and clearly $w=w_1+w_2$ is the sum of two orthogonal minimal tripotents in  $\mathcal{U}_2(w)$, where $w_1=\frac 12(u+w+Q(w)u))$ and  $w_2=\frac 12(-u+w-Q(w)u))$. Finally, having in mind the Peirce arithmetics, we have that $Q(w)u\in \mathcal{U}_0(e)$ and $\{w,w,e\} =P_2(e) \{w,w,e\}=P_2(e)\{w_1,w_1,e\}+P_2(e)\{w_2,w_2,e\}= \frac{1}{2} (\{u,u,e\}+\{w,w,e\})$ which shows that $\{u,u,e\}=\{w,w,e\}$.
\end{proof}

Throughout this paper we will make repeatedly use of Proposition 2.4 and Lemma 2.6 in \cite{BunFerMarPer}, results that we summarize in the following proposition.

\begin{proposition}\label{p normas sub u en Lusin}

Let  $e$ be a tripotent in a JB$^*$-triple $\mathcal{U}$. Let  $x\in \mathcal{U}$ and let $x_j=P_j(e)x$ for $j=1,2$. Then  \begin{itemize}
\item[i)] $P_2(e)\{x,x,e\}\geq 0$ in $\mathcal{U}_2(e)$.

\item[ii)] $\|\{x_j,x_j,e\} \|\leq \|P_2(e)\{x,x,e\} \| $ for $j=1,2$.

\item[iii)] $\| x_j \|^2\leq 4 \| \{x_j,x_j,e\}\|$ for $j=1,2$.
\end{itemize}
\end{proposition}

\section{The minimax principle}\label{sec: minmax}

The main result of this section (Theorem \ref{t minmax}) is a generalization of the minimax principle to the setting of weakly compact JB$^*$-triples. Generalizations of Weyl's inequality and of the Cauchy-Poincaré (interlacing) theorem will be derived from this principle.\medskip

The following results can be considered as orthogonalising Gram-Schmidt processes in JB$^*$-algebras and JB$^*$-triples.\smallskip

\begin{proposition}\label{p sum of n minimal is rank n}
Let $\mathcal{A}$ be a JB$^*$-algebra. For every natural $n$, each collection of $n$ minimal tripotents in  $\mathcal{A}$ is contained in $\mathcal{A}_2(u)$ where $u$ is a tripotent of rank less than or equal to $n$.
\end{proposition}

\begin{proof}
We shall proceed by induction. The case $n=1$ is trivial. Assume that the results holds for a given natural $n\in \mathbb{N}$ and take $n+1$ minimal tripotents $u_1,\ldots,u_{n+1} \in \mathcal{A}$. By the induction hypothesis we may assume that there exists a tripotent $e\in \mathcal{A}$ with rank less than or equal to $n$ such that $\{u_1,\ldots,u_n\}$ is contained in $\mathcal{A}_2(e)$.  By Proposition \ref{p generalised extreme ray property} we have that whenever $\|P_0(e)u_{n+1}\|>0$, $v=\frac{P_0(e)u_{n+1}}{\|P_0(e)u_{n+1}\|}$ is a minimal tripotent in $\mathcal{A}$, $e\perp v$ and $u_{n+1}$ belongs to $\mathcal{U}_2(e+v)$. Thus $u=e+v$ satisfy our desired statement.\smallskip

Suppose now that $\|P_0(e)u_{n+1}\|=0$. Again by Proposition \ref{p generalised extreme ray property} we can assume that $w=\frac{P_1(e)u_{n+1}}{\|P_1(e)u_{n+1}\|}$ is a minimal tripotent in $\mathcal{A}$, otherwise $u_{n+1}=P_2(e)u_{n+1}$ and hence $u=e$ gives our thesis. It can be derived from \cite[Proposition 5.8]{Kau97} that every finite rank tripotent in a JB$^*$-algebra $\mathcal{A}$ is majorized by a unitary in $\mathcal{A}^{**}$. Let $\tilde{e}$ be a unitary in $\mathcal{A}^{**}$ with $e\leq \tilde{e}$ and let us denote $g=\tilde{e}-e$. Clearly $w\in \mathcal{A}^{**}_1(g)$ and  $\{w,w,g\}$ is a positive (non-zero) element in $\mathcal{A}^{**}_2(g)\cap \mathcal{A}^{**}_1(w)$ (see \cite[Lemma 1.5]{FriRus85}) whose range tripotent $v\in \mathcal{A}^{**}_2(g)\cap \mathcal{A}^{**}_1(w)$ has rank one or two (see \cite[Corollary 2.2]{DanFri87}). If $rank(v)=2$ we have that $\{v,w,v\}$ is zero by Peirce arithmetics and a minimal tripotent by \cite[Proposition 2.1]{DanFri87} which gives a contradiction. Therefore $v$ is a minimal tripotent in $\mathcal{A}^{**}_2(g) \cap \mathcal{A}^{**}_1(w)$ and $w\in \mathcal{A}^{**}_1(v)$. Actually, $v$ belongs to $\mathcal{A}$ by the Kadison Transitivity Theorem for JB$^*$-triples given in \cite[Theorem 3.3]{BunFerMarPer}. Finally, $u=e+v$ is a tripotent in $\mathcal{A}$ with rank less than or equal to $n+1$, satisfying that the set $\{u_1,\ldots,u_{n+1}\}$ is contained in $\mathcal{A}_2(u)$.
\end{proof}

\begin{lemma}\label{l pre-minimax}
Let $\mathcal{U}$ be a  JB$^*$-triple. Let  $u_1,\ldots,u_n$ be a collection of $n$ minimal tripotents in $\mathcal{U}$ ($n\in \mathbb{N}$). Then for every tripotent $e$ in $\mathcal{U}$ there exists  $u$ in $\mathcal{U}_2(e)$ a tripotent with $rank(u)\leq n$ such that $$x\in \{ u_1, \ldots, u_n\}^{\perp} \;\;\;\hbox{  for all $x$ in $\mathcal{U}_2(e)\cap \mathcal{U}_0(u)$}. $$
\end{lemma}

\begin{proof}
As a first step we shall prove the case  $n=1$. Assume we have a minimal tripotent $u_1$ and an arbitrary tripotent $e$ in $\mathcal{U}$. We will make use again of the extreme ray property. If $e\perp u_1$ we consider $u=0$ (or any minimal tripotent in $\mathcal{U}_2(e)$ if can be found). If $e$ and $u_1$ are not orthogonal, we define $u=\frac{P_2(e)u_1}{\|P_2(e)u_1\|}$ whenever $\|P_2(e)u_1\|\neq 0$ and $u=\frac{\{u_1,u_1,e\}} {\|\{u_1,u_1,e\}\|} $ in case $\|P_2(e)u_1\|= 0$, which is a minimal tripotent in $\mathcal{U}_2(e)$ by Lemma \ref{l extreme ray prop en Peirce 1} and Proposition \ref{p generalised extreme ray property}. Using Peirce arithmetics our thesis follows straightforwardly.\smallskip

We have shown then that for every minimal tripotent we can find a  tripotent $v$ in $\mathcal{U}_2(e)$ of rank smaller than or equal to 1 such that every element in $\mathcal{U}_2(e)$ orthogonal to $v$ is also ortogonal to our original minimal tripotent.\smallskip

We proceed now with the general case. As we have just seen, for every minimal tripotent $u_i$ we have associated another tripotent $v_i$ in $\mathcal{U}_2(e)$ minimal or zero. We give raise to the  tripotent $u$  by using  Proposition \ref{p sum of n minimal is rank n}, having in mind that $v_1,\ldots,v_n$ are minimal tripotents (or zero) in the JB$^*$-algebra $\mathcal{U}_2(e)$. The desired statement can be checked using Peirce arithmetics.

\end{proof}

Note that the result above  is trivially satisfied when $e$ is a finite rank tripotent of $rank(e)\leq n$ by taking $u=e$. The following corollary is an orthogonalized version of Lemma \ref{l pre-minimax}.

\begin{corollary}\label{c gram schmidt}
Let $\mathcal{U}$ be a  JB$^*$-triple. Let  $f$ be a tripotent in $\mathcal{U}$ with $rank(f)=n$. Then for every tripotent $e$ in $\mathcal{U}$ with $rank(e)=m> n$ there exists  $u$ in $\mathcal{U}_2(e)$ a tripotent with $rank(u)= m-n$ such that $u\perp f$.
\end{corollary}
\bigskip

Every element in a weakly compact JB$^*$-triple can be decomposed as a (possibly infinite) linear combination of mutually orthogonal minimal tripotents. Concretely, given a weakly compact JB$^*$-triple, $\mathcal{U}$, for every  $x \in \mathcal{U}$ there exist a sequence of mutually orthogonal minimal tripotents $(u_i)_{i\in \mathbb{N}}$ and a decreasing  sequence of nonnegative real numbers  $(\lambda_i(x))_{i\in \mathbb{N}}$ such that $ x=\sum \lambda_i(x) u_i$ (\cite[Remark 4.6]{BunChu92}). The real numbers $(\lambda_i(x))$ are called the \emph{singular values} of $x$, and they are precisely the singular values of the function $x$ when considered as an element of (the commutative C$^*$-algebra) $\mathcal{U}_x$, i.e. the set of all the singular values coincide with the triple spectrum of $x$. We will refer to the sum $ \sum \lambda_i(x) u_i$ as an \emph{atomic decomposition} of $x$. When the singular values vanishes at some natural (and thereafter) we can choose to remove all these summands or just consider the corresponding tripotents to be zero.  The reader should be aware that while the sequence of singular values is uniquely determined, there are  many different choices of the minimal tripotents the moment we have a singular value with multiplicity greater than one.\medskip

\begin{lemma}\label{l innequality in minimax theorem}
Let $\mathcal{U}$ be a weakly compact JB$^*$-triple. Let $x$ be an element in $\mathcal{U}$ with an atomic decomposition $x= \sum_{i\geq 1} \lambda_i (x) v_i$. Then, for every natural $n$ and for every $e\in \mathcal{U}$ tripotent of rank $n$ there exists a minimal tripotent $v$ in $\mathcal{U}_2(e)$ such that $$\|x\|_v\leq \lambda_n(x). $$
\end{lemma}
\begin{proof}
For $n=1$ we have that $\|x\|_v\leq\|x\|=\lambda_1(x)$, for every minimal tripotent $v\in \mathcal{U}$.

Assume now that $n\geq 2$. Take $e$ an arbitrary tripotent in $\mathcal{U}$ of rank $n$. We have that $\{v_i: i=1,\ldots,n-1\}$ is a family of mutually orthogonal minimal tripotents and Corollary  \ref{c gram schmidt} assures the existence of a minimal tripotent $v$ in  $\mathcal{U}_2(e)$ such that $v\in \{ v_1, \ldots, v_{n-1}\}^{\perp}$. Since $v\perp v_i$ we have that $\|v_i\|_{v}=0$ ($i=1,\ldots,n-1$)  and $$\|x\|_{v}\leq \sum_{i=1}^{n-1} \|\lambda_i(x) v_i\|_{v} + \|\sum_{i\geq n}\lambda_i(x) v_i\|_{v}\leq \|\sum_{i\geq n}\lambda_i(x) v_i \|=\lambda_n(x).$$
\end{proof}

The following result is the announced generalization of the minimax principle to the setting of weakly compact JB$^*$-triples.\smallskip

\begin{theorem}\label{t minmax}
Let $x$ be an element in a weakly compact JB$^*$-triple $\mathcal{U}$. Then for every natural $n$ we have \begin{align}\label{eq minmax principle}\lambda_n(x) &=\max_e \inf \{\|x\|_{v}: v\in \mathcal{U}_2(e) \hbox{ minimal tripotent} \} \\ &=\min_f \sup\{ \|x\|_{v}: v\in \mathcal{U}_0(f) \hbox{ minimal tripotent}\} \notag \end{align} where $e$ (respectively, $f$) runs over the set of tripotents in $\mathcal{U}$ of rank $n$ (respectively, $n-1$) and $\lambda_n(x)$ is the $n$-th singular value of $x$.

\end{theorem}

\begin{proof}

We will prove first that $\lambda_n(x) =\displaystyle\max_e \inf \{\|x\|_{v}: v\in \mathcal{U}_2(e) \hbox{ minimal tripotent} \}$ where $e$ runs over the set of tripotents in $\mathcal{U}$ of rank $n$. To this end, fix an atomic decomposition of the element $x=\sum \lambda_i(x) v_i$ and $n\in \mathbb{N}$. Fix also  $e=\sum_{i=1}^{n} v_i$ that is a tripotent in $\mathcal{U}$ of rank $n$. For every minimal tripotent $v$ in  $\mathcal{U}_2(e)$  we have that $$ \|x\|_{v}^2= \| P_2(v)\{x,x,v\|= \|P_2(v)\{\sum_{i=1}^{n} \lambda_i(x) v_i, \sum_{i=1}^{n} \lambda_i(x) v_i,v\} \| = $$ $$\sum_{i=1}^{n} \lambda_i(x)^2 \|P_2(v)\{v_i,v_i,v\}\|\geq  \lambda_n(x)^2 \sum_{i=1}^{n} \|v_i\|_{v}^2=\lambda_n(x)^2  \|\sum_{i=1}^{n} v_i\|_{v}^2=$$ $$  \lambda_n(x)^2 \|e\|_{v}^2=\lambda_n(x)^2 \|P_2(v)\{e,e,v\}\|=\lambda_n(x)^2 \|v\|=\lambda_n(x)^2.$$ Thus $\displaystyle \lambda_n(x) \leq \max_e \inf \{\|x\|_{v}: v\in \mathcal{U}_2(e) \hbox{ minimal tripotent} \}  $.\smallskip

The reverse inequality is a direct consequence of Lemma \ref{l innequality in minimax theorem}.\smallskip

We will show now that $\displaystyle \lambda_n(x) =\min_f \sup\{ \|x\|_{v}: v\in \mathcal{U}_0(f) \hbox{ minimal tripotent}\}$  where $f$ runs over the set of tripotents in $\mathcal{U}$ of rank $n-1$. In case $n$ equals one, the result is trivial, since the only tripotent $f$ with rank zero is $f=0$ thus $\mathcal{U}_0(f)=\mathcal{U}$, and  $\|x\|_{v}\leq \|x\|=\|x\|_{v_1}=\lambda_1(x)$ for every minimal tripotent $v$ in $\mathcal{U}$. Assume henceforth that $n\geq 2$. For $f=v_1+\ldots+v_{n-1}$ we have that, for every $v$ minimal tripotent in $\mathcal{U}_0(f)$, $$\|x\|_{v}=\|\sum_{i\geq n} \lambda_i(x) v_i \|_{v} \leq \|\sum_{i\geq n} \lambda_i(x) v_i \| = \|x\|_{v_n}=\lambda_n(x), $$ which shows $\displaystyle\lambda_n(x) \geq \min_f \sup\{ \|x\|_{v}: v\in \mathcal{U}_0(f) \hbox{ minimal tripotent}\}$.\smallskip

On the other hand consider $f$ an arbitrary tripotent in $\mathcal{U}$ with $rank(f)=n-1$. We define $g=v_1+\ldots+v_n$ a tripotent in $\mathcal{U}$ with $rank(g)=n$. By Corollary \ref{c gram schmidt} there exists a (minimal) tripotent $v\in \mathcal{U}_2(g)$ which is orthogonal to $f$, thus $v$ belongs to $\mathcal{U}_0(f)$. Having in mind that $v\in \mathcal{U}_2(g)$ is  also orthogonal to every $v_i$ with $i>n$ we have that $$ \|x\|_{v}=\|\sum_{i=1}^n \lambda_i(x) v_i \|_{v} \geq \|\sum_{i=1}^n \lambda_n(x) v_i \|_{v} = \lambda_n(x)\|g \|_{v} = \lambda_n(x),$$ which finishes the proof.
\end{proof}

The following result is a generalization of the classical Weyl's inequality obtained in \cite{Wey}. More concretely, we shall show that the distance between the corresponding singular values of two elements in a weakly compact JB$^*$-triple is bounded by the distance between the elements.\medskip

\begin{theorem}\label{t Weyl's Inequality distance between singular values}
Let $\mathcal{U}$ be a weakly compact JB$^*$-triple. Let $x, y$ be two elements in  $\mathcal{U}$ with singular values  $(\lambda_i(x))_{i\in \mathbb{N}}$ and $(\lambda_i(y))_{i\in \mathbb{N}}$, respectively. Then \begin{equation}\label{eq Weyl inequality}  \sup_{n\in \mathbb{N}} |\lambda_n(x) - \lambda_n(y)|\leq \|x-y\|.\end{equation}

\end{theorem}

\begin{proof}
By Theorem \ref{t minmax} we have that whenever $e$ runs over the set of tripotents in $\mathcal{U}$ of rank $n$,  $$\lambda_n(x) =\max_e \inf\{\|x\|_{v}: v\in \mathcal{U}_2(e) \hbox{ minimal tripotent}  \}=$$  $$\max_e \inf \{\|y-(y-x)\|_{v}: v\in \mathcal{U}_2(e) \hbox{ minimal tripotent} \}\leq $$ $$\max_e \inf \{\|y\|_{v}: v\in \mathcal{U}_2(e) \hbox{ minimal tripotent} \} + \|y-x\|=\lambda_n(y)+\|y-x\|$$ and our statement follows straightforwardly.
\end{proof}

Given a JB$^*$-triple, $\mathcal{U}$, a tripotent $e$ in $\mathcal{U}$  and a minimal tripotent  $v$ in $ \mathcal{U}_2(e) \cup \mathcal{U}_0(e)$  it should be clear that $\|P_2(e)x\|_v, \|P_0(e)x\|_v \leq \|x\|_v$. Therefore a direct application of the minimax theorem allows us to state the following  Interlacing theorem or Cauchy-Poincaré theorem for weakly compact JB$^*$-triples (see for example \cite[Theorem 2.1]{Bha} for the matricial case).

\begin{theorem}\label{t Interlacing theorem of Cauchy-Poincare}
Let $\mathcal{U}$ be a weakly compact JB$^*$-triple. Let $x$ be an element in $\mathcal{U}$ with singular values $(\lambda_n(x))_{n\in \mathbb{N}}$. Then for every tripotent $e$ in $\mathcal{U}$ each singular value of $P_2(e)x$ and $P_0(e)x$ is bounded by the corresponding singular value of $x$, i.e. $$ \lambda_n(P_2(e)x), \lambda_n(P_0(e)x) \leq \lambda_n(x). $$ $\hfill\Box$
\end{theorem}

We finish this section by presenting another characterization of the $n$-th singular value  of an element in terms of the distance to the set of elements of rank $n$ (see \cite[Chapter II. Theorem 2.1]{GohKre} for compact operators on a Hilbert space). When particularized to the space of compact operators on a Hilbert space, which is a weakly compact JB$^*$-triple (every compact C$^*$-algebra is), we get exactly the just quoted result.\smallskip

\begin{theorem}\label{t other characterization of singular values}
Let $\mathcal{U}$ be a weakly compact JB$^*$-triple, $x\in  \mathcal{U}$. Then $$\lambda_{n}(x)=\min \{\|x-a\|: rank(a)\leq n-1 \}.$$

\end{theorem}

\begin{proof}
Fix an atomic decomposition, $\sum_{i\geq 1} \lambda_i (x) v_i$, of $x$.\smallskip

Let us consider a finite rank element in $a$ in $\mathcal{U}$ with $rank(a)=n-1$. We denote by $f_a$ the range tripotent of $a$ which is a finite rank tripotent of rank $n-1$ ($f_a=0$ when $n=1$). For every minimal tripotent $v\in \mathcal{U}_0(f_a)$ we have that $a\perp v$ and, by the minimax principle (\ref{eq minmax principle}), $$\lambda_{n}(x) \leq \sup\{ \|x\|_{v}: v\in \mathcal{U}_0(f_a) \hbox{ minimal tripotent}\} =$$ $$ \sup\{ \|x-a\|_{v}: v\in \mathcal{U}_0(f_a) \hbox{ minimal tripotent}\}\leq \|x-a\|.$$

On the other hand, taking $a=\sum_{i=1}^{n-1}\lambda_i(x)v_i$ ($a=0$ when $n=1$), we have that $\|x-a\|=\lambda_n(x)$ which finishes the proof.
\end{proof}

\section{Ky Fan maximum principle} \label{sec: Ky Fan}

 The classical Ky Fan maximum principle \cite[Theorem 1]{Fan49} states that given an hermitian operator in $B(H)$, where $H$ is a finite dimensional Hilbert space, for every natural $n$ we have that $$ \sum_{i=1}^{n} \lambda_i(A)=\max \sum_{i=1}^{n} <Ax_i,x_i>, $$ when $n$ orthogonal vectors $x_i$ ($i=1,\ldots,n$) vary in $H$.\smallskip

In order to extend this result to weakly compact JB$^*$-triples, we may think, probably dazzled by the minimax principle given in Theorem \ref{t minmax}, that we  only have to change the values $<Ax_i,x_i>$ with the seminorms $\|A\|_{u_i}$ associated to minimal tripotents. The following example highlights that such a direct generalization can not be obtained.\smallskip

\begin{example}\label{ex a trivial Ky Fan fails}
Let $\mathcal{U}$ be the C$^*$-algebra of $2\times 2$ matrices over the complex numbers. Given $\alpha, \beta, \delta \in \mathbb{R}$, with $\alpha \delta=\beta^2 \neq 0$ and $\alpha^2+2\beta^2+\delta^2=1$, we have that the matrix $A=\begin{pmatrix}  \alpha & \beta \\ \beta & \delta   \end{pmatrix}$ is a minimal projection in $\mathcal{U}$ with eigenvalues $\lambda_1(A)=1$ and $\lambda_2(A)=0$. For $u_1=\begin{pmatrix}  1 & 0 \\ 0 & 0  \end{pmatrix}$ and $u_2=\begin{pmatrix}  0 & 0 \\ 0 & 1  \end{pmatrix}$, which are orthogonal minimal tripotents (actually projections) in  $\mathcal{U}$, we have that $$ \lambda_1(A)+\lambda_2(A)=1 < \sqrt{\alpha^2+\beta^2}+\sqrt{\delta^2+\beta^2} = \|A\|_{u_1}+\|A\|_{u_2}.$$

\end{example}

The latter example shows that a selection among the families of mutually orthogonal minimal tripotents has to be done if we pretend to extend Ky Fan maximum principle to the Jordan setting.\medskip

We introduce now families of ($p$-Schatten) seminorms associated to mutually orthogonal minimal tripotents in a JB$^*$-triple. This kind of seminorms (in the case $p=2$) were previously introduced in \cite{PerRod01} with more generality.\smallskip

\begin{definition}
Let $U$ be a JB$^*$-triple. Let $u_1,\ldots,u_n$ be mutually orthogonal minimal tripotents in $\mathcal{U}$. We define the following family of seminorms $$\| x \|_{p,u_1,\ldots,u_n} = (\sum_{i=1}^n \|x\|_{u_i}^p)^{\frac{1}{p}} \hspace{1cm} (1\leq p)$$

\end{definition}

Clearly all these seminorms coincide when $n=1$. This families of seminorms are closely related to symmetric gauge functions (see \cite{Mir60}). \smallskip

The following result follows straightforwardly from Lemma \ref{l innequality in minimax theorem} by induction.

\begin{lemma}\label{l p-seminorms bounded by eigenvalues}
Let $U$ be a JB$^*$-triple. Given an element $x\in \mathcal{U}$, for every natural $n\in \mathbb{N}$ and for every finite rank tripotent $e$ in $\mathcal{U}$  with $rank(e)=n$, we have that $$ \inf \{ \|x\|_{p,u_1,\ldots,u_n} \} \leq (\sum_{i=1}^n \lambda_i(x)^p)^{\frac{1}{p}},$$ where $\{u_1,\ldots,u_n\}$ runs over the families of mutually orthogonal minimal tripotents in  $\mathcal{U}_2(e)$. $\hfill\Box$
\end{lemma}

We will pay a special attention to the seminorms $\| \cdot \|_{2,u_1,\ldots,u_n}$.\smallskip

Our next result, which in particular characterizes  (geometrically) when a minimal tripotent belongs to the Peirce-2 subspace associated to a finite rank tripotent, will be extremely useful throughout this section.\smallskip

\begin{proposition}\label{p finite traces p=2}
Let $\mathcal{U}$ be a JB$^*$-triple and let $u_1,\ldots,u_n$ be mutually orthogonal minimal tripotents in $\mathcal{U}$. Then for every minimal tripotent $v$ in $\mathcal{U}$ we have $$\| v \|_{2,u_1,\ldots,u_n} \leq 1. $$ Moreover,  $\| v \|_{2,u_1,\ldots,u_n} = 1 $ if and only if $v$ belongs to $\mathcal{U}_2(u_1+\ldots+u_n)$. If that is the case we also have  $\| v \|_{2,\tilde{u}_1,\ldots,\tilde{u}_n} = 1 $ for every collection $\{\tilde{u}_1,\ldots,\tilde{u}_n\}$ of mutually orthogonal minimal tripotents in $\mathcal{U}_2(u_1+\ldots+u_n)$.
\end{proposition}
\begin{proof}
We shall proceed by induction. The case $n=1$ is trivial.

Assume now that the statement is satisfied for some natural $n$. Take a family of mutually orthogonal minimal tripotents in $\mathcal{U}$, $\{u_1,\ldots,u_n,u_{n+1}\}$.\smallskip

We will show first that $\| v \|_{2,u_1,\ldots,u_{n+1}} = 1$ for every minimal tripotent $v$ in $\mathcal{U}_2(u_1+\ldots+u_{n+1})$. Define $e=u_1+\ldots+u_n$. By Proposition \ref{p generalised extreme ray property} we can decompose $v$ as a combination of tripotents:\smallskip

\emph{Case 1}\smallskip

There exists $ \alpha, \beta, \gamma, \delta \in \mathbb{C}$ with $|\alpha|^2+|\beta|^2 + |\gamma|^2+ |\delta|^2=1$, $\alpha \delta  = \beta \gamma$, and minimal tripotents $v_{11}\in \mathcal{U}_2(e)$, $v_{12}, v_{21}\in \mathcal{U}_1(e)\cap \mathcal{U}_1(u_{n+1})$, $v_{22}\in  \mathcal{U}_2(u_{n+1})$  satisfying that $(v_{11},v_{12},v_{21},v_{22})$ is a quadrangle and $v = \alpha v_{11}+\beta v_{12}+\gamma v_{21}+\delta v_{22}$. By Lemma \ref{l extreme ray prop en Peirce 1} there exist two minimal tripotents $\tilde{v}_{12}$, $\tilde{v}_{21}$ in $ \mathcal{U}_2(e)$ such that $2\{v_{12},v_{12},e\} = \{\tilde{v}_{12},\tilde{v}_{12},e\}$ and $2\{v_{21},v_{21},e\} = \{\tilde{v}_{21},\tilde{v}_{21},e\}$ from which we deduce that $\| v_{12} \|_{2,u_1,\ldots,u_n}^2 = \frac{1}{2} \| \tilde{v}_{12} \|_{2,u_1,\ldots,u_n}^2$, $\| v_{21} \|_{2,u_1,\ldots,u_n}^2=\frac{1}{2}\| \tilde{v}_{21} \|_{2,u_1,\ldots,u_n}^2$. Moreover, by Peirce arithmetics it should be clear that $\|v_{12}\|_{u_{n+1}}^2=\frac{1}{2}=\|v_{21}\|_{u_{n+1}}^2$ and $\|v_{22}\|_{u_{n+1}}^2=1$. Therefore, applying the induction hypothesis, $$\| v \|_{2,u_1,\ldots,u_{n+1}}^2 = \sum_{i=1}^{n+1}  \|v\|_{u_i}^2=|\alpha|^2\sum_{i=1}^{n}  \|v_{11}\|_{u_i}^2+\frac{|\beta|^2}{2}\sum_{i=1}^{n}  \|\tilde{v}_{12}\|_{u_i}^2+$$ $$\frac{|\gamma|^2}{2}\sum_{i=1}^{n}  \|\tilde{v}_{12}\|_{u_i}^2+ |\beta|^2 \|v_{12}\|_{u_{n+1}}^2+|\gamma|^2 \|v_{21}\|_{u_{n+1}}^2+|\delta|^2 \|v_{22}\|_{u_{n+1}}^2=$$ $$|\alpha|^2+\frac{|\beta|^2}{2} + \frac{|\gamma|^2}{2} + \frac{|\beta|^2}{2}+ \frac{|\gamma|^2}{2}+ |\delta|^2 = 1.$$\smallskip

\emph{Case 2}\smallskip

There exists $ \alpha, \beta, \delta \in \mathbb{C}$ with $|\alpha|^2+2|\beta|^2 + |\delta|^2=1$, $\alpha \delta  = \beta^2$, minimal tripotents $v_{11}\in \mathcal{U}_2(e)$, $v_{22}\in  \mathcal{U}_2(u_{n+1})$ and a tripotent $w\in \mathcal{U}_1(e)\cap \mathcal{U}_1(u_{n+1})$ satisfying that $(v_{11},w,v_{22})$ is a trangle and $v = \alpha v_{11}+\beta w+\delta v_{22}$. By Lemma \ref{l extreme ray prop en Peirce 1} there exists a minimal tripotent $\tilde{w}$, in $ \mathcal{U}_2(e)$ such that $\{w,w,e\} = \{\tilde{w},\tilde{w},e\}$ and hence $\| w \|_{2,u_1,\ldots,u_n}= \| \tilde{w} \|_{2,u_1,\ldots,u_n}$. Having in mind that $\|w\|_{u_{n+1}}^2=1=\|v_{22}\|_{u_{n+1}}^2$  we have that, by the induction hypothesis, $$\| v \|_{2,u_1,\ldots,u_{n+1}}^2 = \sum_{i=1}^{n+1}  \|v\|_{u_i}^2=|\alpha|^2\sum_{i=1}^{n}  \|v_{11}\|_{u_i}^2+|\beta|^2\sum_{i=1}^{n}  \|\tilde{w}\|_{u_i}^2+$$
$$ |\beta|^2 \|w\|_{u_{n+1}}^2+|\delta|^2 \|v_{22}\|_{u_{n+1}}^2=|\alpha|^2+2|\beta|^2 + |\delta|^2 = 1.$$\medskip

We will prove next the statement for a general minimal tripotent $v\in \mathcal{U}$. Define now $e=u_1+\ldots +u_{n+1}$. By Proposition \ref{p generalised extreme ray property} and Lemma \ref{l extreme ray prop en Peirce 1} and using the same arguments  given above we get  that $\| v \|_{2,u_1,\ldots,u_{n+1}}^2 = |\alpha|^2+\frac{|\beta|^2}{2} + \frac{|\gamma|^2}{2}=1-\frac{|\beta|^2}{2} - \frac{|\gamma|^2}{2}-|\delta|^2\leq 1$ in the first case and $\| v \|_{2,u_1,\ldots,u_{n+1}} ^2= |\alpha|^2+|\beta|^2=1-|\beta|^2-|\delta|^2\leq 1$ in the second one.\medskip

The implication $\| v \|_{2,u_1,\ldots,u_{n+1}} = 1\Rightarrow v\in \mathcal{U}_2(u_1+\ldots+u_{n+1})$ follows from the equivalence $v\in \mathcal{U}_2(u_1+\ldots+u_{n+1}) \Leftrightarrow \beta=\gamma=\delta=0.$
\end{proof}

We recall that an $n\times n$ matrix over the real numbers is said to be \emph{doubly-stochastic} whenever all the entries are non-negative and the sum of the elements of every row or column is 1.\smallskip

\begin{corollary}\label{c p=2 doubly-stochastic}
Let $\mathcal{A}$ be a finite-rank JB$^*$-algebra with $rank(\mathcal{A})=n$. Then for every pair of frames in $\mathcal{A}$ (maximal families of mutually orthogonal minimal tripotents) $\{v_1,\ldots,v_n\}$ and $\{u_1,\ldots,u_n\}$ the matrix $$ (a_{ij})_{i,j\in\{1,\ldots,n\}} \hspace{0.5cm} \hbox{ where } a_{ij}=\|v_i\|_{u_j}^2  $$ is doubly-stochastic.
\end{corollary}
\begin{proof}
Clearly $a_{ij}\geq 0$ $\forall i,j \in \{1,\ldots,n\}$.
For any fixed $i_0\in \{1,\ldots,n\}$ we have that $\sum_{j=1}^{n} a_{i_0j}=\sum_{j=1}^{n} \|v_{i_0}\|_{u_j}^2 = \| v_{i_0} \|_{2,u_1,\ldots,u_n}^2 = 1 $, by Proposition \ref{p finite traces p=2}. On the other hand, given a fixed $j_0\in \{1,\ldots,n\}$ and having in mind that $u_{j_0}$ is a minimal tripotent and $\sum_{i=1}^{n}v_i$ is unitary in $\mathcal{A}$,  we have that $$\sum_{i=1}^{n} a_{ij_0}=\sum_{i=1}^{n} \|v_{i}\|_{u_{j_0}}^2 = \sum_{i=1}^{n} \|P_2(u_{j_0})\{v_i,v_i,u_{j_0}\}\| =  \| \sum_{i=1}^{n}P_2(u_{j_0})\{v_i,v_i,u_{j_0}\}\| = $$ $$ \| P_2(u_{j_0})\{\sum_{i=1}^{n}v_i,\sum_{i=1}^{n}v_i,u_{j_0}\}\| = \| P_2(u_{j_0})u_{j_0}\| = \| u_{j_0}\|=1.$$

\end{proof}

The following result is in fact contained as a particular case in the main result of this section (Theorem \ref{t Ky Fan maximum principle} below). However, we include it and its proof here as an appetizer.

\begin{theorem}\label{t Ky Fan principle p=2}
Let $ \sum_{i\geq 1} \lambda_i(x) v_i$ be an atomic decomposition of an element $x$ in a weakly  compact JB$^*$-triple $\mathcal{U}$. Then for every natural $n$  we have $$ \max_{e}  \inf \{ \|x\|_{2,u_1,\ldots,u_n} \} = (\sum_{i=1}^n \lambda_i(x)^2)^{\frac{1}{2}},$$ where $\{u_1,\ldots,u_n\}$ runs over the families of mutually orthogonal minimal tripotents in  $\mathcal{U}_2(e)$ and $e$ is a finite rank tripotent in $\mathcal{U}$ with $rank(e)=n$.

\end{theorem}
\begin{proof}
By Lemma \ref{l p-seminorms bounded by eigenvalues} we only have to prove that $\displaystyle  \max_{e}  \inf \{ \|x\|_{2,u_1,\ldots,u_n} \}$ is greater than or equal to $ (\sum_{i=1}^n \lambda_i(x)^2)^{\frac{1}{2}}.$ Take $\displaystyle x=\sum_{i\geq 1} \lambda_i(x) v_i$ an atomic decomposition of $x$ and define $e=v_1+\ldots+v_n$. Since $v_i\perp e$  ($i\geq n+1$), for every minimal tripotent $u_j$ in $\mathcal{U}_2(e)$ we have that $\|x\|_{u_{j}}^2 = \sum_{i=1}^{n}{\lambda_i(x)}^2 \|v_i\|_{u_j}^2$. Therefore, given $\{u_1,\ldots,u_n\}$ any family of mutually orthogonal minimal tripotents in $\mathcal{U}_2(e)$ we have that $$\|x\|_{2,u_1,\ldots,u_n}^2=\sum_{i=1}^{n} \lambda_i(x)^2 (\sum_{j=1}^{n} \|v_i\|_{u_j}^2)= \sum_{i=1}^{n} \lambda_i(x)^2 \|v_i\|_{2,u_1,\ldots,u_{n}}^2 = \sum_{i=1}^{n} \lambda_i(x)^2$$ by Proposition \ref{p finite traces p=2}.
\end{proof}

The following technical result is the final step towards our main result.\smallskip

\begin{lemma}\label{l innequality in ky-fan minimax}
Let $ \sum_{i\geq 1} \lambda_i v_i$ be an atomic decomposition of an element $x$ in a weakly  compact JB$^*$-triple $\mathcal{U}$. Fix a natural $n$ and let us define $e=v_1+\ldots+v_n$. Then for every frame $\{u_1,\ldots,u_n\}$ in $\mathcal{U}_2(e)$ we have that $$(\sum_{i=1}^n \lambda _i^p)^{\frac{1}{p}}\leq \|x\|_{p,u_1,\ldots,u_n} \leq \frac{n^{\frac{1}{p}}}{\sqrt{n}} (\sum_{i=1}^{n} \lambda_i^2)^{\frac{1}{2}} \hspace{0.5 cm} (p\geq 1).$$ In particular $\displaystyle \sum_{i=1}^{n} \lambda_i \leq \|x\|_{1,u_1,\ldots,u_n} \leq \sqrt{n} (\sum_{i=1}^{n} \lambda_i^2)^{\frac{1}{2}}$.
\end{lemma}

\begin{proof}
For a fixed frame $\{u_1,\ldots,u_n\} $ in $\mathcal{U}_2(e)$ we denote $a_j=\|\sum_{i= 1}^{n} \lambda_i v_i\|_{u_j}$, so we have that $$\|x\|_{p,u_1,\ldots,u_n}^p=   \|P_2(e)x\|_{p,u_1,\ldots,u_n}^p= \| \sum_{i= 1}^{n} \lambda_i v_i\|_{p,u_1,\ldots,u_n}^p= \sum_{j= 1}^{n}a_j^p.$$ By (the proof of) Theorem \ref{t Ky Fan principle p=2} it  could be observed that $\sum_{j=1}^n a_j^2=\sum_{i= 1}^{n} \lambda_i^2$. It is also clear that $a_j\leq \lambda_1$ for every $j=1,\ldots,n$. Therefore we are are dealing with an optimization problem for a function $f_1:A_1=[0,\lambda_1]^n \to \mathbb{R}$ defined by $f_1(x_1,\ldots,x_n)=\sum_{j=1}^n x_j^p$ with the constrain given by the function  $g_1(x_1,\ldots,x_n)=\sum_{j=1}^n x_j^2-\sum_{i= 1}^{n} \lambda_i^2.$ The maximum of this function is attained when all the values are equal, i.e. at $x=\left((\frac{\sum_{i= 1}^{n} \lambda_i^2}{n})^{\frac{1}{2}},\ldots,(\frac{\sum_{i= 1}^{n} \lambda_i^2}{n})^{\frac{1}{2}}\right )$ which gives the desired upper bound. The minimum of $f_1$ is attained when one of the coordinates is $\lambda_1$, so we can assume that $a_1= x_1=\lambda_1$. We want to apply a (natural) induction argument but before we have to check the following fact:\medskip

\emph{Claim}\smallskip

Let $k_1 =\min \{k\in \{1,\ldots,n\}: \lambda_k<\lambda_1\}$, and suppose that $a_j = \lambda_1$ for every $j<k_1$. Then $v_k\perp u_j$ and $a_k\leq \lambda_{k_1}$ for every $j< k_1 \leq k$.\medskip

Since, for every $j< k_1 \leq k$, $$\lambda_1^2=a_j^2=\|x\|_{u_j}^2 \leq \|x\|_{u_j}^2+(\lambda_1^2-\lambda_k^2)\|v_k\|_{u_j}^2\leq \|\sum_{i= 1}^{n} \lambda_1 v_i\|_{u_j}^2\leq \lambda_1^2,$$ we have that $\|v_k\|_{u_j}^2=0$ and hence, using Proposition \ref{p normas sub u en Lusin}, $v_k \perp u_j$. Now it is clear that $a_k=\|\sum_{i= k_1}^{n} \lambda_i v_i\|_{u_j}\leq \lambda_{k_1}$ proving our claim.\medskip

Define now $f_2:A_2=[0,\lambda_2]^{n-1} \to \mathbb{R}$ given by $f_2(x_2,\ldots,x_n)=\sum_{j=2}^n x_j^p$. Notice that we have $a_j\leq \lambda_2$ for every $j\geq 2$ (trivially when $\lambda_1=\lambda_2$ or using our claim above when $\lambda_1<\lambda_2=\lambda_{k_0}$). Now we have the constrain given by $g_2(x_2,\ldots,x_n)=\sum_{j=2}^n x_j^2-\sum_{i= 2}^{n} \lambda_i^2$. Again the minimum of $f_2$  is attained when one of the coordinates is $\lambda_2$ and the induction process should be clear from this point on.\smallskip

Therefore our original function $f_1$ attains its minimum (except for a rearrangement of the variables) at the element $(\lambda_1,\ldots, \lambda_n)$ which finishes the proof.

\end{proof}

The final result of this section is a generalization in the Jordan setting of the Ky Fan maximum principle.\smallskip
\begin{theorem}\label{t Ky Fan maximum principle}
Let $ \sum_{i\geq 1} \lambda_i(x) v_i$ be an atomic decomposition of an element $x$ in a weakly  compact JB$^*$-triple $\mathcal{U}$. Then for every natural $n$  we have $$ \max_{e}  \min \{ \|x\|_{p,u_1,\ldots,u_n} \} = (\sum_{i=1}^n \lambda_i(x)^p)^{\frac{1}{p}} \hspace{0.5 cm} (p\geq 1),$$ where $\{u_1,\ldots,u_n\}$ runs over the families of mutually orthogonal minimal tripotents in  $\mathcal{U}_2(e)$ and $e$ is a finite rank tripotent in $\mathcal{U}$ with $rank(e)=n$.
\end{theorem}
\begin{proof}

The inequality $\displaystyle \max_{e}  \min \{ \|x\|_{p,u_1,\ldots,u_n} \} \leq (\sum_{i=1}^n \lambda_i(x)^p)^{\frac{1}{p}}$ is given by Lemma \ref{l p-seminorms bounded by eigenvalues}.\smallskip

The reverse inequality is given by Lemma \ref{l innequality in ky-fan minimax} when we consider the frame $\{v_1,\ldots,v_n\}$ in $\mathcal{U}_2(e)$ with $e=v_1+\ldots+v_n$.
\end{proof}

\section{Perturbation of spectral resolutions} \label{sec: perturbation of eigenvectors}

This section we deal with perturbations of elements in a weakly compact JB$^*$-triple. Given a weakly compact JB$^*$-triple $\mathcal{U}$ a perturbation of an element $x\in \mathcal{U}$ is another element $y\in \mathcal{U}$ with $\|x-y\|\leq \varepsilon$ where $0<\varepsilon$ is commonly assumed to be small. Weyl's inequality (\ref{eq Weyl inequality}) assures the continuity of the  corresponding $n$-th eigenvalue of $x$ and $y$ for every natural $n$. However, the distance between the corresponding minimal tripotents appearing in some atomic decompositions of $x$ and $y$ (eigenvectors for selfadjoint elements in the C$^*$-algebra case) can be large (see for example \cite[page 46]{BhaDavMcI}). This disappointing situation might be caused by the choice of the atomic decomposition or because the elements $x$ and $y$ are not closed enough to each other.\smallskip

The first obstacle is solved by considering spectral resolutions instead of minimal tripotents, i.e.  given an element $x$ in a weakly compact JB$^*$-triple, $\mathcal{U}$, we will  express $x$ as the sum  $\sum_{i=1}^{\infty} \sigma_i(x) e_i$ where $\{\sigma_i(x): i\in \mathbb{N}\}$ are the singular values of $x$  taken in (strictly) decreasing order not counting multiplicity (the eigenvalues of the positive function $x$ when considered as an element in $\mathcal{U}_x\cong C_0(Sp(x))$) and $e_i$ is the characteristic function of the set $\{\sigma_i(x)\}$ in $\mathcal{U}_x$ \cite[Remark 4.6]{BunChu92}. We will call this sum the \textit{spectral decomposition of} $x$ and contrary to the case of the atomic one, this decomposition is unique (see for example the discussion appearing in \cite[Proposition 3.6]{DanFri87} in the case of the spin factor). It is also well-known that each $e_i$ is a  finite rank tripotent in $\mathcal{U}$ (see for example \cite[proof of Proposition 4.5]{BunChu92}), whose rank coincides with the multiplicity of $\{\sigma_i(x)\}$.\smallskip

To overcome the second obstacle is the main goal of this section which culminates in Theorem \ref{t continuity of spectral resolutions} where we show, for every natural $n$, a connection (continuity) between the norm $\|x-y\|$ and the distance between the corresponding spectral resolutions of $x$ and $y$.\medskip

We shall begin obtaining some technical results on the distance between tripotents in general JB$^*$-triples.

\begin{lemma}\label{l distance between tripotents}
Let $\mathcal{U}$ be a JB$^*$-triple and let $e,f$ be tripotents in $\mathcal{U}$. Let us set $\delta=\|e-P_2(e)f\| $. Then
 $$\|P_1(e)f\|\leq 2\sqrt{2}\sqrt{\delta}\;, \hspace{0.5cm} \|P_0(f)e\| \leq 8\delta+16\sqrt{2}\sqrt{\delta}.$$
Moreover, defining $\tilde{\delta}= \|f-P_2(f)e\|$ we have that
$$\|e-f\|\leq \delta+2\sqrt{2}\sqrt{\delta}+8\tilde{\delta}+16\sqrt{2}\sqrt{\tilde{\delta}}.$$
\end{lemma}
\begin{proof}
By adding and subtracting $e$ it is easy to see that \begin{equation}\label{eq lemma distance 1}  \|e-\{P_2(e)f,P_2(e)f,e\}\| \leq 2\delta  \hbox{ and } \|e-Q(P_2(e)f)^2e\|\leq 4\delta.\end{equation}

Having in mind Proposition \ref{p normas sub u en Lusin}, Peirce arithmetics and $\|\{f,f,e\}\|\leq 1$, we have that  $0\leq \{P_2(e)f,P_2(e)f,e\} +\{P_1(e)f,P_1(e)f,e\} =  P_2(e)\{f,f,e\}\leq e$, where the order is the one given by the JB$^*$-algebra structure of $\mathcal{U}_2(e)$. We deduce that $\|\{P_1(e)f,P_1(e)f,e\}\| \leq \| e-\{P_2(e)f,P_2(e)f,e\} \|\leq 2\delta$ and, again by Proposition \ref{p normas sub u en Lusin}, we have that \begin{equation}\label{eq lemma distance 2} \|P_1(e)f\|\leq 2\sqrt{2}\sqrt{\delta}. \end{equation}

In order to get the second inequality we will split the proof into pieces. First  of all we have that $\{f,f,e\} = \{f,P_1(e)f,e\} +\{P_2(e)f,P_2(e)f,e\} +\{P_1(e)f,P_2(e)f,e\} $ by Peirce arithmetics. Therefore, by (\ref{eq lemma distance 1}), (\ref{eq lemma distance 2}) and the continuity of the norm, we have that  \begin{equation}\label{eq lemma distance 3} \|e-\{f,f,e\}\|\leq 2\delta + 4\sqrt{2}\sqrt{\delta}. \end{equation} Analogously, denoting $z=\{P_2(e)f,e,P_2(e)f\}$, we have  $Q(f)^2e = \{f,\{f,e,f\},f\} = \{f ,\{f, e,P_1(e)f\} , f\}+\{f ,\{P_1(e)f, e,P_2(e)f\} , f\}+\{f,z,f\}$ and also $ \{f,z,f\} = 2\{P_1(e)f,z,f\} + Q(P_2(e)f)^2e.$ Again by (\ref{eq lemma distance 1}) and (\ref{eq lemma distance 2}) we get \begin{equation}\label{eq lemma distance 4} \|e-Q(f)^2e\}\|\leq 4\delta + 8\sqrt{2}\sqrt{\delta}. \end{equation} Finally, from (\ref{eq lemma distance 3}) and (\ref{eq lemma distance 4}) we derive that \begin{equation}\label{eq lemma distance 5} \|P_0(f)e\|=\|e-2\{f,f,e\}+Q(f)^2e \| = \|(2e-2\{f,f,e\})+Q(f)^2e-e \| \leq 8\delta + 16\sqrt{2}\sqrt{\delta}. \end{equation} The last inequality is direct from $e-f=(e-P_2(e)f)-P_1(e)f-P_0(e)f$.
\end{proof}

When we only consider finite rank tripotents of the same rank, the hypothesis of the above Lemma can be relaxed.\smallskip

\begin{lemma}\label{l distance between tripotents same finite rank}
Let $e,f$ be finite rank tripotents in a JB$^*$-triple, $\mathcal{U}$, with $rank(e)=rank(f)=m$. Let us set $\delta=\|e-P_2(e)f\| $. Then $$\|e-f\|\leq (m+1) \delta+ 2\sqrt{2}\sqrt{\delta}.$$
\end{lemma}
\begin{proof}
By hypothesis $f$ is the sum of $m$ mutually orthogonal minimal tripotents, $f=\sum_{i=1}^{m} u_i$. Fix $i_0\in \{1,\ldots,m\}$. By Proposition \ref{p sum of n minimal is rank n} and  Corollary \ref{c gram schmidt} applied to $\{P_2(e)u_i: i\neq i_0\}$ and $e$,  we can find a minimal tripotent $v$  in $\mathcal{U}_2(e)$ which is orthogonal to every $P_2(e)u_i$ with $i\neq i_0$. Having in mind that $\|x+y\|_v=\|x\|_v$ for every $y\perp v$ we have that $$\|P_2(e)u_{i_0}\|\geq \|P_2(e)u_{i_0}\|_{v}=\|P_2(e)f\|_{v}=\|e-(e-P_2(e)f)\|_{v}\geq $$ $$ \|e\|_{v}-\|e-P_2(e)f\|_{v}\geq \|e\|_{v}-\|e-P_2(e)f\|\geq  1-\delta,$$ and by the extreme ray property (Proposition \ref{p generalised extreme ray property}) we get  $\|P_0(e)u_{i_0}\|\leq \delta$. Since ${i_0}$ was arbitrarily chosen we have that $\|P_0(e)f\|\leq \|P_0(e)u_1\|+\ldots+\|P_0(e)u_m\|\leq m \delta$. By Lemma \ref{l distance between tripotents} we also have that $\|P_1(e)f\|\leq 2\sqrt{2}\sqrt{\delta}$ and hence $$\|e-f\|\leq \|e-P_2(e)f \|+\| P_1(e)f\|+\|P_0(e)f \|\leq (m+1) \delta+ 2\sqrt{2}\sqrt{\delta}.$$
\end{proof}

Our last technical result exhibits the continuity of the Peirce projections associated to tripotents. The proof is merely an exercise so we only include a brief sketch of it.

\begin{lemma}\label{l continuity of peirce projections}
Let $\mathcal{U}$ be a JB$^*$-triple. Let $e,f$ be tripotents in $\mathcal{U}$ with $\|e-f\|=\delta$. Then the following inequalities hold:
\begin{itemize}
\item[(\rm{a)}] $ \|P_2(e)-P_2(f)\|\leq 4 \delta$, $ \|P_1(e)-P_1(f)\|\leq 12 \delta$, $ \|P_0(e)-P_0(f)\|\leq 8 \delta$.
\item [(\rm{b)}] $\|P_k(u)P_j(v)\|\leq 8\delta$ where $u, v \in \{e,f\}$ distinct and $k,j\in \{0,1,2\}$ are distinct.
\end{itemize}
 In particular, given a norm-one element $x\in \mathcal{U}$, satisfying $x=e+P_0(e)x$ we have that $ \|P_1(f)x\|,\|P_2(f)x-f\|,\|x- (f+P_0(f)x)\| \leq 9\delta.$
\end{lemma}

\begin{proof}

(\rm{a)} Use the defining identities of the Peirce projections $P_2(e)=Q(e)^2$, $P_1(e)=2(L(e,e)-Q(e)^2)$ and $P_0(e)=Id-2L(e,e)+Q(e)^2.$\smallskip

(\rm{b)} Since $(P_k(u)-P_k(v))P_j(v)=P_k(u)P_j(v)=P_k(u)(P_j(v)-P_j(u))$ and $k\neq j$ (one of them is different from 1) we just apply \rm{(a)}.\smallskip

Finally, given $x\in \mathcal{U}$ a norm one element with $x=e+P_0(e)x$, we have that $P_1(f)x = P_1(f)(e-f)+P_1(f)P_0(e)x$, $P_2(f)x-f = P_2(f)(e-f)+P_2(f)P_0(e)x$ and $x-(f+P_0(f)x) = (e-f)+(P_0(e)-P_0(f))x$ thus by (\rm{a)} and (\rm{b)} we are done.
\end{proof}

We will next generalize to the setting of JB$^*$-algebras a result of C. Davis for compact selfadjoint elements in (the C$^*$-algebra) $B(H)$ (see \cite[Theorem 2.1]{Dav}).\smallskip

Take a selfadjoint  element, $a$,  in a weakly compact JB$^*$-algebra $\mathcal{J}$. Let $\beta\geq 0, \gamma >0$ and assume that $p$ is the spectral resolution of $a$ associated to the set $[\nu, \mu]$ where $\mu-\nu= 2\beta$ and the sets $]\nu-\gamma,\nu[$, $]\mu,\mu+\gamma[$ contains no eigenvalues of $a$. Given another selfadjoint element $b\in \mathcal{J}$ with $\|b-a\|\leq \delta<\frac{\gamma}{2}$ we say that $q$ is the projection of $b$ associated to $p$ (where $q$ is the spectral resolution of $[\nu-\delta, \mu+\delta])$ and the sets $]\nu-\gamma+\delta,\nu-\delta[$, $]\mu+\delta,\mu+\gamma-\delta[$ contains no eigenvalues of $b$. Moreover, by Weyl's inequality (\ref{eq Weyl inequality}),  $p$ and $q$ have the same rank. We will maintain this notation in the following result.

\begin{theorem}\label{t Davis JB}
Let $\mathcal{J}$ be a weakly compact JB$^*$-algebra. Let $a,b$ be selfadjoint elements in $\mathcal{J}$ and let $p,q$ be associated spectral projections in $\mathcal{J}$. Then $$\|p-P_2(p)q\|=\|P_2(p)(1-q)\|\leq\frac{(\beta+\delta)^2}{(\beta+\gamma-\delta)^2}.$$ Moreover, whenever $p$ is a finite rank projection with $rank(p)=m$, denoting $\alpha= \frac{(\beta+\delta)^2}{(\beta+\gamma-\delta)^2} $, we also have that $$\|p-q\|\leq (m+1) \alpha+ 2\sqrt{2}\sqrt{\alpha}.$$
\end{theorem}

\begin{proof}
It is not restrictive to consider $a,b$ weakly compact elements in a unital JBW$^*$-algebra (consider $\mathcal{J}^{**}$ instead of $\mathcal{J}$) and $[\nu, \mu]=[-\beta,\beta]$ just changing $a,b$ by  $a-\lambda I, b-\lambda I$, respectively,  with $\lambda = \frac{1}{2}(\mu+\nu)$. Having in mind that $b=P_2(q)b+P_2(1-q)b$, we have that $b^2=\{b,b,1\}=\{b,b,q\}+\{b,b,1-q\}$ and $\{b,b,q\}=P_2(q)b^2=(P_2(q)b)^2$, $\{b,b,1-q\}=P_2(1-q)b^2=P_0(q)b^2=(P_0(q)b)^2$ are  positive elements in $\mathcal{J}_2(q)$ and $\mathcal{J}_2(1-q)$, respectively. Moreover, since (the absolute value of) the eigenvalues of $P_2(1-q)b$ are all greater than $\beta+\gamma-\delta$, we have that $\{b,b,1-q\}=(P_2(1-q)b)^2\geq (\beta+\gamma-\delta)^2 (1-q)$. Therefore, $$P_2(p)b^2=P_2(p) \{b,b,q\}+P_2(p) \{b,b,1-q\}\geq P_2(p) \{b,b,1-q\}\geq $$ $$ (\beta+\gamma-\delta)^2 P_2(p)(1-q),$$ and hence $$ \| P_2(p)(1-q) \|\leq \frac{1}{ (\beta+\gamma-\delta)^2}\| P_2(p)b^2\|.$$

On the other hand, having in mind that $a=P_2(p)a+P_0(p)a$, $\|P_2(p)a\| \leq \beta$ and $\|a-b\|\leq \delta$, we have that  $$\| P_2(p)b^2 \| = \| P_2(p)(\{b,b,p\}+\{b,b,1-p\})\| = \| P_2(p)(\{b,b,p\} \|= $$ $$\| P_2(p)  (\{a,a,p\}+\{b-a,a,p\}+\{a,b-a,p\}+\{b-a,b-a,p\})\| \leq $$ $$\|   \{P_2(p)a,P_2(p)a,p\}\|+ 2 \|P_2(p)a\| \|b-a\| +\|b-a\|^2\leq (\beta+\delta)^2,$$ which shows $\|p-P_2(p)q\|\leq\frac{(\beta+\delta)^2}{(\beta+\gamma-\delta)^2}.$\smallskip

The final assertion follows from Lemma \ref{l distance between tripotents same finite rank}.
\end{proof}

The classical Davis' Theorem, and its generalization to JB$^*$-algebras above, becomes pleasingly useful when we consider spectral projections associated to eigenvalues (i.e. $\beta=0$), giving a bound on the distance of the corresponding spectral projections when we consider finite dimensional C$^*$-algebras. This fact, together with  the polar decomposition of an element in a C$^*$-algebra, was used in \cite[Theorem 3.6]{BecFer} to exhibit the continuity (at some fixed point) of the perturbed spectral resolutions. The lack of a polar decomposition in the category of JB$^*$-algebras appears as an (at first sight) unsolvable obstruction in order to generalize \cite[Theorem 3.6]{BecFer} to the Jordan setting. However, the arguments given in the proof of Theorem \ref{t Davis JB} turned out to be somehow inspiring.\smallskip

Let $x$ be an element in a weakly compact JB$^*$-triple $\mathcal{U}$.  Let $\sum_{i\geq 1} \sigma_i(x) e_i$ be the spectral decomposition of $x$ and fix  a natural number $n\in \mathbb{N}$. Let $\delta$ be a positive number satisfying  $\delta < \frac 12\min \{\sigma_i-\sigma_{i+1}: i=1,\ldots,n \}$. Given $y\in \mathcal{U}$ with $\|x-y\|\leq \delta$,  associated to each tripotent, $e_i$, we have another tripotent $f_i$, the spectral resolution of $y$ with respect to the set $[\sigma_i(x)-\delta,\sigma_i(x)+\delta]$, for every $i\in \{1,\ldots,n\}$. Moreover, the rank of $e_i$ and $f_i$ coincide by Weyl's inequality (\ref{eq Weyl inequality}). We will keep this notation in the following result which is the main theorem of this section and exhibits the continuity at some fixed point of the perturbed spectral resolutions.\smallskip

\begin{theorem}\label{t continuity of spectral resolutions}
Let $\mathcal{U}$ be a weakly compact JB$^*$-triple. Let $ x$ be an element in $\mathcal{U}$. Given $n\in \mathbb{N}$ and $\varepsilon>0$ there exists $\delta>0$ such that for every $y$ with $\|x-y\|\leq\delta$, $\|e_i-f_i\|\leq\varepsilon $ for each $i\in \{1,\ldots,n\}$.

\end{theorem}
\begin{proof}

We will prove first the case $n=1$.\smallskip

Fix $\varepsilon>0$. Let $\sum_{i\geq 1} \sigma_i e_i$ be the spectral decomposition of $x$, $m_1=rank(e_1)$. Let us consider the function $h:\mathbb{R}^+\to \mathbb{R}$ defined by $h(t)=\frac{2t}{\sigma_1}+\frac{4\sqrt{2}  \sqrt{t^2+4t \sigma_1}}{ \sqrt{3} (\sigma_1-\sigma_2)}$, which is increasing and  continuous with $\lim_{t\to 0} h(t)=0$. Let $t_1\in \mathbb{R}^+$ such that $(m_1+1)h(t)+2\sqrt{2} \sqrt{h(t)} \leq \varepsilon,$ for every $0<t\leq t_1$.\smallskip

Choose a positive $\delta$ satisfying \begin{equation}\label{eq theorem continuity of spectral resolutions 1} \delta < \min \{\frac{\sigma_1-\sigma_2}{4}, t_1 \}.\end{equation} Given $y\in \mathcal{U}$ with $\|x-y\|\leq \delta$ we consider $\sum_{i\geq 1} \lambda_i v_i$ an atomic decomposition of $y$. Then  we have that $\lambda_1,\ldots,\lambda_{m_1} \in [\sigma_1-\delta,\sigma_1+\delta]$, $\lambda_1-\lambda_{m_1+1}\geq \frac{\sigma_1-\sigma_2}{2}$ and the spectral resolution of $y$ associated to $e_1$ is $f_1=v_1+\ldots+v_{m_1}$. \smallskip

First of all we have that $\delta \geq \| P_2(e_1) x- P_2(e_1)y\|=\| \sigma_1 e_1-P_2(e_1)y\|=\sigma_1 \| e_1-\frac{1}{\sigma_1}P_2(e_1)y\|$ thus arguing as in (\ref{eq lemma distance 1}) we have \begin{equation}\label{eq theorem continuity of spectral resolutions 2}\|e_1 - \frac{1}{\sigma_1^2}\{P_2(e_1)y,P_2(e_1)y,e_1\}\|\leq \frac{2\delta}{\sigma_1}.\end{equation} Therefore $$(1-\frac{2\delta}{\sigma_1})e_1\leq \frac{1}{\sigma_1^2} \{P_2(e_1)y,P_2(e_1)y,e_1\} \leq $$ $$\frac{1}{\sigma_1^2} P_2(e_1) \{y,y,e_1\} = \frac{1}{\sigma_1^2} \sum_{i\geq 1} \lambda_i^2 P_2(e_1) \{v_i,v_i,e_1\}\leq $$ $$  \frac{1}{\sigma_1^2} \sum_{i\geq 1} \lambda_i^2 P_2(e_1) \{v_i,v_i,e_1\} +\frac{1}{\sigma_1^2} \sum_{i\geq m_1+1} (\lambda_1^2-\lambda_{m_1+1}^2) P_2(e_1) \{v_i,v_i,e_1\} \leq $$ $$ \frac{\lambda_1^2}{\sigma_1^2} \sum_{i\geq 1} P_2(e_1) \{v_i,v_i,e_1\} =  \frac{\lambda_1^2}{\sigma_1^2} P_2(e_1) \{\sum_{i\geq 1}v_i,\sum_{i\geq 1}v_i,e_1\}\leq $$ $$\frac{\lambda_1^2}{\sigma_1^2} e_1\leq (1+\frac{\delta^2}{\sigma_1^2}+ \frac{2\delta}{\sigma_1})e_1,$$ where in the last inequality we are using that $\lambda_1\in [\sigma_1-\delta,\sigma_1+\delta]$.  The reader should be aware that although the orthogonal sum $\sum_{i\geq 1}v_i$ is not, in general, an element in $\mathcal{U}$ it is a tripotent in $\mathcal{U}^{**}$ when considered as the limit of the partial sums in the weak$^*$-topology (see \cite[Corollary 3.13]{Horn87}). We will work with this kind of elements in $\mathcal{U}^{**}$ without any other explicit mention.\smallskip

We deduce from  these inequalities that \begin{equation}\label{eq theorem continuity of spectral resolutions 3}\frac{\lambda_1^2-\lambda_{m_1+1}^2}{\sigma_1^2} \sum_{i\geq m_1+1}  P_2(e_1) \{v_i,v_i,e_1\}\leq \frac{\delta^2+4\delta \sigma_1}{\sigma_1^2} e_1,\end{equation} and hence $$P_2(e_1) \{\sum_{i\geq m_1+1} \lambda_i v_i,\sum_{i\geq m_1+1} \lambda_i v_i,e_1\} \leq $$ $$ \lambda_{m_1+1}^2  P_2(e_1) \{\sum_{i\geq m_1+1}v_i,\sum_{i\geq m_1+1}v_i,e_1\} \leq \frac{\lambda_{m_1+1}^2}{\lambda_1^2-\lambda_{m_1+1}^2} (\delta^2+4\delta \sigma_1) e_1.$$ By Proposition \ref{p normas sub u en Lusin} we have that \begin{equation}\label{eq theorem continuity of spectral resolutions 4} \|P_2(e_1) \sum_{i\geq m_1+1} \lambda_i v_i  \| \leq 2\sqrt{\frac{\lambda_{m_1+1}^2(\delta^2+4\delta \sigma_1)}{\lambda_1^2-\lambda_{m_1+1}^2}  } = \frac{2 \lambda_{m_1+1} \sqrt{\delta^2+4\delta \sigma_1}}{ \sqrt{\lambda_1^2-\lambda_{m_1+1}^2}}.\end{equation} Using again that $\lambda_i\in [\sigma_1-\delta,\sigma_1+\delta]$ for every $i\in \{1,\ldots,m_1\}$ we have that $\|\sum_{i=1}^{m_1} \frac{\lambda_i}{\sigma_1} v_i-\sum_{i=1}^{m_1}  v_i \|\leq \frac{\delta}{\sigma_1}$, thus $$\|e_1-P_2(e_1)f_1\|\leq \|P_2(e_1) ( \sum_{i=1}^{m_1} \frac{\lambda_i}{\sigma_1} v_i-\sum_{i=1}^{m_1}  v_i) \| + \|e_1-\frac{1}{\sigma_1}P_2(e_1) y \|+ $$ $$\|\frac{1}{\sigma_1} P_2(e_1) \sum_{i\geq m_1+1} \lambda_i v_i \| \leq \frac{2\delta}{\sigma_1}+\frac{2 \lambda_{m_1+1} \sqrt{\delta^2+4\delta \sigma_1}}{ \sigma_1 \sqrt{\lambda_1^2-\lambda_{m_1+1}^2}}.$$

Now, having in mind that  $\lambda_{m_1+1} \leq \sigma_1$, $\lambda_1 \geq \frac{3}{4} \sigma_1$  and $\lambda_1-\lambda_{m_1+1}\geq \frac{\sigma_1-\sigma_2}{2}$ we deduce that \begin{equation}\label{eq theorem continuity of spectral resolutions 5}\lambda_1^2-\lambda_{m_1+1}^2 = (\lambda_1-\lambda_{m_1+1})(\lambda_1+\lambda_{m_1+1})\geq \frac{(\sigma_1-\sigma_2)}{2} \frac{3\sigma_1}{4}\geq \frac{3}{2}(\frac{\sigma_1-\sigma_2}{2})^2,\end{equation} and hence \begin{equation}\label{eq theorem continuity of spectral resolutions 6} \|e_1-P_2(e_1)f_1\|\leq \frac{2\delta}{\sigma_1}+\frac{4\sqrt{2}  \sqrt{\delta^2+4\delta \sigma_1}}{ \sqrt{3} (\sigma_1-\sigma_2)}=h(\delta).\end{equation}

Finally, by Lemma \ref{l distance between tripotents same finite rank} and the restrictions on  $\delta$ given in (\ref{eq theorem continuity of spectral resolutions 1}) we have that $$\|e_1-f_1\|\leq (m_1+1)h(\delta)+2\sqrt{2} \sqrt{h(\delta)} \leq\varepsilon.$$\smallskip

Now we proceed with the general case by induction.\medskip

Assume $n\geq 2$ and fix $\varepsilon>0$. Again, let $\sum_{i\geq 1} \sigma_i e_i$ be the spectral decomposition of $x$ and let  $m_k=rank(e_1+\ldots+e_k)$, $k=1,\ldots,n$. Clearly $m_k-m_{k-1}$ is the rank of the tripotent $e_k$. Let us define $\gamma_k=\min\{\sigma_i-\sigma_{i+1}: i = 1,\ldots,k\}$ for every $k$ in $\{1,\ldots,n\}$.\smallskip

We consider the (increasing and  continuous) function $h_n:\mathbb{R}^+\to \mathbb{R}$ defined by $h_n(t) = \frac{7t \sigma_1}{\sigma_n} + \frac{4\sqrt{2}  \sqrt{(6t \sigma_1)^2+4(6t \sigma_1) \sigma_n^2}}{ \sigma_n \sqrt{3} (\sigma_n-\sigma_{n+1})}$, which satisfies $\lim_{t\to 0} h_n(t)=0$. Let $t_n\in \mathbb{R}^+$ such that $(m_n-m_{n-1}+1)h_n(t)+2\sqrt{2} \sqrt{h_n(t)} \leq \varepsilon,$ for every $0<t\leq t_n$.\smallskip

By the induction hypothesis, associated to the positive $t_n$, there exists $\tilde{\delta}>0$ (with $\tilde{\delta}<\gamma_{n-1}$) such that, for every $y\in \mathcal{U}$ with $\|x-y\|\leq \tilde{\delta}$, we have that $\|(e_1+\ldots+e_{n-1})-(f_1+\ldots+f_{n-1}) \|\leq t_n $.\smallskip

Choose a positive $\delta$ satisfying \begin{equation}\label{eq theorem continuity of spectral resolutions 7} \delta < \min \{\frac{\gamma_n}{4}, t_n \sigma_1, \frac{ 6 t_n \sigma_1}{\sigma_n},\tilde{\delta}\}.\end{equation} Given $y\in \mathcal{U}$ with $\|x-y\|\leq \delta$ we consider $\sum_{i\geq 1} \lambda_i v_i$ an atomic decomposition of $y$. \smallskip

Having in mind that $\|(e_1+\ldots+e_{n-1})-(f_1+\ldots+f_{n-1}) \|\leq t_n $ and $\|y\|\leq \|x\|+\delta\leq \frac{5}{4} \sigma_1$, we deduce from Lemma \ref{l continuity of peirce projections} that $$ \|P_2(e_n) P_2(f_1+\ldots+f_{n-1})y\|= \|P_2(e_n) (P_2(f_1+\ldots+f_{n-1})-P_2(e_1+\ldots+e_{n-1}))y\| \leq $$ $$ \|(P_2(f_1+\ldots+f_{n-1})-P_2(e_1+\ldots+e_{n-1}))y\| \leq 4\|y\| t_n \leq 5 t_n \sigma_1.$$ Denoting by $z=P_0(f_1+\ldots + f_{n-1})y$, since $\sigma_n e_n-P_2(e_n)z=P_2(e_n)(x-y)+P_2(e_n) P_2(f_1+\ldots+f_{n-1})y$, we get \begin{equation}\label{eq theorem continuity of spectral resolutions 8} \|\sigma_n e_n-P_2(e_n)z\| \leq \delta +5 t_n \sigma_1 \leq 6 t_n \sigma_1. \end{equation}

From this point we can reproduce the same arguments given from equation (\ref{eq theorem continuity of spectral resolutions 2}) to equation (\ref{eq theorem continuity of spectral resolutions 6}) obtaining subsequently,
\begin{equation}\label{eq theorem continuity of spectral resolutions 9}\|e_n - \frac{1}{\sigma_n^2}\{P_2(e_n)z,P_2(e_n)z,e_n\}\|\leq 2\frac{ 6 t_n \sigma_1}{\sigma_n},\end{equation}
\begin{equation}\label{eq theorem continuity of spectral resolutions 10}\frac{\lambda_{m_{n-1}+1}^2-\lambda_{m_n+1}^2}{\sigma_n^2} \sum_{i\geq m_n+1}  P_2(e_n) \{v_i,v_i,e_n\}\leq \frac{(\frac{ 6 t_n \sigma_1}{\sigma_n})^2+4\frac{ 6 t_n \sigma_1}{\sigma_n}\sigma_n}{\sigma_n^2} e_n,\end{equation}
\begin{equation}\label{eq theorem continuity of spectral resolutions 11} \|P_2(e_n) \sum_{i\geq m_n+1} \lambda_i v_i  \| \leq  \frac{2 \lambda_{m_n+1} \sqrt{(\frac{ 6 t_n \sigma_1}{\sigma_n})^2+4 \frac{ 6 t_n \sigma_1}{\sigma_n}\sigma_n}}{ \sqrt{\lambda_{m_{n-1}+1}^2-\lambda_{m_n+1}^2}}.\end{equation}
Since $\lambda_{m_n+1} \leq \sigma_n$, $\lambda_{m_{n-1}+1} \geq \frac{3}{4} \sigma_n$  and $\lambda_{m_{n-1}+1}-\lambda_{m_n+1}\geq \frac{\sigma_n-\sigma_{n+1}}{2}$ we also deduce that \begin{equation}\label{eq theorem continuity of spectral resolutions 12}\lambda_{m_{n-1}+1}^2-\lambda_{m_n+1}^2 \geq \frac{(\sigma_n-\sigma_{n+1})}{2} \frac{3\sigma_n}{4}\geq \frac{3}{2}\left(\frac{\sigma_n-\sigma_{n+1}}{2}\right)^2,\end{equation} and having in mind the conditions on $\delta$ given in (\ref{eq theorem continuity of spectral resolutions 7}) we get
\begin{equation}\label{eq theorem continuity of spectral resolutions 13} \|e_n-P_2(e_n)f_n\|\leq \frac{\delta}{\sigma_n}+\frac{ 6 t_n \sigma_1}{\sigma_n} + \frac{4 \sqrt{2} \sqrt{(\frac{6 t_n \sigma_1}{\sigma_n})^2+4  (6 t_n \sigma_1) } }{ \sqrt{3} (\sigma_n-\sigma_{n+1})} \leq \end{equation} $$ \frac{ 7 t_n \sigma_1}{\sigma_n} + \frac{4 \sqrt{2} \sqrt{(\frac{6 t_n \sigma_1}{\sigma_n})^2+4  (6 t_n \sigma_1) } }{ \sqrt{3} (\sigma_n-\sigma_{n+1})} = h_n(t_n).$$

Again, by Lemma \ref{l distance between tripotents same finite rank}, and the conditions of $h_n$ and $t_n$, we have that $$\|e_n-f_n\|\leq (m_n-m_{n-1}+1)h_n(t_n)+2\sqrt{2} \sqrt{h_n(t_n)} < \varepsilon.$$

\end{proof}

For later purposes we would extract a particular case from Theorem \ref{t continuity of spectral resolutions}.

\begin{remark}\label{r perturbation of supports}
Let $\mathcal{U}$ be a weakly compact JB$^*$-triple. Given $x$ a norm-one element in $\mathcal{U}$ we denote by  $e=s(x)$ its support tripotent  (i.e. $\sigma_1(x)=1$ and $e=e_1$), $\gamma =1 -\|x-e\|$ (i.e. $\gamma=\sigma_1(x)-\sigma_2(x)$) and $m=rank(e)$. Let $\delta$ be a positive number with $\delta<\frac{\gamma}{4}$ and suppose that $y\in \mathcal{U}$ satisfies $\|x-y\|\leq \delta$. Denoting by $f$ the spectral resolution of $y$ corresponding to the set $[1-\delta,1+\delta]$ we have that $$\|e-f\|\leq (m+1) \left(2\delta +\frac{4\sqrt{2}  \sqrt{\delta^2+4\delta }}{ \sqrt{3} \gamma}\right)+2\sqrt{2} \left(2\delta +\frac{4\sqrt{2}  \sqrt{\delta^2+4\delta }}{ \sqrt{3} \gamma}\right)^{\frac{1}{2}} .$$
\end{remark}

\section{Perturbation of convex combinations} \label{sec: perturbation of convex combinations}

In \cite{AbrBecHalLimPol}, the authors introduce the following geometric property in the general setting of Banach spaces, where $B(x,\delta)$ denotes the closed ball centered at $x$ with radius $\delta$.

\begin{definition}\label{d (co)}

A Banach space, $X$, is said to have the property $\hbox{\rm{(co)}}$ if for every $n\in \mathbb{N}$, given $x_1,\ldots,x_n \in \mathcal{B}_X$, $\lambda_1,\ldots,\lambda_n>0$ with $\sum_{i=1}^{n} \lambda_i=1$ and $\varepsilon >0$ there exist $\delta>0$ and continuous functions $\Phi_i:B(x_0,\delta)\cap \mathcal{B}_X\to B(x_i,\varepsilon)\cap \mathcal{B}_X$, where $x_0=\sum_{i=1}^{n} \lambda_i x_i$, satisfying $y=\sum_{i=1}^{n} \lambda_i\Phi_i(y)$ for  every $y\in B(x_0,\delta)$.

\end{definition}\medskip

It was shown in \cite{BecFer} that finite dimensional C$^*$-algebras have property $\hbox{\rm{(co)}}$. The main result of this section (Theorem \ref{t weakly compact triples are (co)}) assures that the same holds for weakly compact JB$^*$-triples. We will generalize to the setting of JB$^*$-triples some technical lemmas appearing in \cite{BecFer}, concretely Lemma 3.1 and Lemma 3.2. These results, together with those appearing in Section \ref{sec: perturbation of eigenvectors}, will allow us to prove that weakly compact JB$^*$-triples have property $\hbox{\rm{(co)}}$.\smallskip

\begin{lemma}\label{l norm-control of convex combinations}
Let $\mathcal{U}$ be a JB$^*$-triple. Let $x$, $y$ be two elements in $\mathcal{B}_{\mathcal{U}}$  with $d = \|x+y\|$. Then for every $\lambda\in [0,\frac 12]$ we have $$\|\lambda x+(1-\lambda)y\|\leq\sqrt{1-(4-d^2)(\lambda-\lambda^2)} \leq 1-\frac{(4-d^2)\lambda}{4}$$
\end{lemma}

\begin{proof}
It is well-known (see for example \cite[Proposition 1]{FriRus86}) that $\|x+y\|=\sup\{ (\|x+y\|_\phi: \phi \in \partial _e \mathcal{B}_{\mathcal{U}^*}\}$ where the supremum is taken on the set of extreme points  of the unit sphere of  $\mathcal{U}^*$, $\partial _e \mathcal{B}_{\mathcal{U}^*}$. Therefore, for every $ \phi \in \partial _e \mathcal{B}_{\mathcal{U}^*}$ with support tripotent $s(\phi)\in \mathcal{U}^{**}$ (see \cite[Proposition 4]{FriRus85}), we have  $$4\geq d^2 \geq \|x+y\|_\phi^2=\|x\|_\phi^2 +\|y\|_\phi^2+ \phi(\{x,y,s(\phi)\}+\{y,x,s(\phi)\}),$$ thus $$\phi(\{x,y,s(\phi)\}+\{y,x,s(\phi)\}\leq d^2-\|x\|_\phi^2 -\|y\|_\phi^2.$$

Now, it is straightforward to verify that   $$\|\lambda x+(1-\lambda) y\|^2 =  \sup \{ \|\lambda x+(1-\lambda) y\|^2_\phi :\phi \in \partial _e \mathcal{B}_{\mathcal{U}^*}\} =$$ $$  \sup \{\lambda ^2 \|x\|_\phi^2+ (1-\lambda)^2 \|y\|_\phi^2 +\lambda (1-\lambda)\phi(\{x,y,s(\phi)\}+\{y,x,s(\phi)\}) :\phi \in \partial _e \mathcal{B}_{\mathcal{U}^*}\}\leq$$ $$ \sup\{\lambda ^2 \|x\|_\phi^2+ (1-\lambda)^2 \|y\|_\phi^2 + \lambda (1-\lambda) (d^2-\|x\|_\phi^2 -\|y\|_\phi^2) : \phi\in \partial _e \mathcal{B}_{\mathcal{U}^*}\} = $$  $$\sup\{(\lambda ^2-\lambda (1-\lambda)) \|x\|_\phi^2+ ((1-\lambda)^2-\lambda (1-\lambda)) \|y\|_\phi^2 + \lambda (1-\lambda) d^2 : \phi\in \partial _e \mathcal{B}_{\mathcal{U}^*}\} \leq$$  $$  (\lambda ^2-\lambda (1-\lambda))+ ((1-\lambda)^2-\lambda (1-\lambda))  + \lambda (1-\lambda) d^2= 1-(4-d^2)(\lambda-\lambda ^2).$$ The last inequality follows from the facts  $\sqrt{1+t}\leq 1+\frac{t}{2}$ for $t\geq-1$ and $\lambda-\lambda ^2\geq \frac{\lambda}{2}. $
\end{proof}

Given  $x$ a norm-one element in a JB$^*$-triple, $\mathcal{U}$, there exists a (unique) non-zero tripotent in $\mathcal{U}^{**}$, denoted by $s(x)$, such that $x=s(x)+P_0(s(x))x$ (see \cite[Lemma 3.3]{EdwRut88} or \cite[Page 130]{EdwFerHosPer}). We will call this tripotent the \emph{support tripotent} of $x$ (in $\mathcal{U}^{**}$).

\begin{lemma}\label{l support of convex combinations}
Let $\mathcal{U}$ be a JB$^*$-triple and let $x$, $y$ be two elements in $\mathcal{B}_{\mathcal{U}}$. Then there exists a tripotent $e$ in $\mathcal{U}^{**}$ satisfying $$ \lambda x+ (1-\lambda)y= e+ P_0(e)(\lambda x+ (1-\lambda)y) \hbox{ for all } \lambda \in ]0,1[, $$ and being maximal for this property. \smallskip

In particular, when $\mathcal{U}$ is a weakly compact  JB$^*$-triple we have that $e$ is a finite rank tripotent in $\mathcal{U}$ and $\|P_0(e)(\lambda x+ (1-\lambda)y)\|<1$ for all $ \lambda \in ]0,1[$.
\end{lemma}

\begin{proof}
Fix $\lambda_0\in ]0,1[$ and set $e=s(\lambda_0 x+(1-\lambda_0)y)\in \mathcal{U}^{**}$ whenever $\|\lambda_0 x+(1-\lambda_0)y\|=1$ and $e=0$ in other case. Assume first that  $\|\lambda_0 x+(1-\lambda_0)y\|=1$ and  $e=s(\lambda_0 x+(1-\lambda_0)y)$.\smallskip

Since $\mathcal{F}_e=(e+\mathcal{U}_0^{**}(e))\cap \mathcal{B}_{\mathcal{U}}$ is a norm-closed face in the closed unit ball of $\mathcal{U}$ (see \cite{EdwRut96}) we have that  $\lambda x+ (1-\lambda)y \in \mathcal{F}_e$ for every $\lambda \in [0,1]$. Therefore $s(\lambda x+ (1-\lambda)y)\geq e=s(\lambda_0 x+(1-\lambda_0)y)$ for every $\lambda \in ]0,1[$. The arbitrariness of $\lambda_0$ gives $s(\lambda x+ (1-\lambda)y)=e$ for every $\lambda \in ]0,1[$ and the maximality of $e$. \smallskip

Having in mind the above arguments, the case  $e=0$ is now trivial. \smallskip

The final comments come from the fact that for every  norm-one element $x$ in a weakly compact  JB$^*$-triple, $1$ is an isolated point in the triple spectrum of $x$ (see \cite{BunChu92}).
\end{proof}

When in the proof of \cite[Theorem 3.8]{BecFer} Lemma 3.2, Remark 3.7, Lemma 3.3, Lemma 3.1 and Theorem 3.6 are replaced with Lemma \ref{l support of convex combinations}, Remark \ref{r perturbation of supports}, Lemma \ref{l continuity of peirce projections}, Lemma \ref{l norm-control of convex combinations} and Theorem \ref{t continuity of spectral resolutions} respectively, the same proof given in \cite[Theorem 3.8]{BecFer}, with some minor modifications, applies to give our final result. We include a proof for the sake of completeness.\smallskip

\begin{theorem}\label{t weakly compact triples are (co)}
Every weakly compact JB$^*$-triple has property $\hbox{\rm{(co)}}$.
\end{theorem}\medskip

\begin{proof}

Let $\mathcal{U}$ be a weakly compact  JB$^*$-triple, $n \in \mathbb{N}$ . Let $x_1, \ldots, x_n$ be elements in the closed unit ball of $\mathcal{U}$ and $\lambda_1,\ldots,\lambda_n>0$  with $ \sum_{i=1}^{n} \lambda_i=1$. We claim that for every positive $\varepsilon$ there exist a positive $\delta$ such that given $y\in \mathcal{B}_\mathcal{U}$ with $ \|y-\sum_{i=1}^{n} \lambda_i x_i\|\leq \delta$, there exist $\tilde{x}_1,\ldots, \tilde{x}_n$ in $ \mathcal{B}_\mathcal{U}$ satisfying $ y=\sum_{i=1}^{n} \lambda_i \tilde{x}_i$ and $\|x_i-\tilde{x}_i \|\leq \varepsilon$, for each $i\in \{1,\ldots,n\}$.\medskip

Assume first that $\|\displaystyle \sum_{i=1}^{n} \lambda_i x_i\|=1$.\medskip

In this case, by Lemma \ref{l support of convex combinations}, denoting by $e$ the support tripotent of $ \sum_{i=1}^{n} \lambda_i x_i$, we have that $e\neq 0$ and $e$ is also the support tripotent of any other (strict) convex combination of the elements $\{x_1,\ldots,x_n\}$.\smallskip

We set, for each $j\in \{1,\ldots,n\}$, $\displaystyle a_j=\sum_{i=1, i\neq j}^{n} \frac{x_i}{n-1} \in \mathcal{B}_{\mathcal{U}}$. We define $d=\max \{\|P_0(e)(a_j+x_j)\|: j\in \{1,\ldots,n\}\}$. It should be clear from Lemma \ref{l support of convex combinations} that $d<2$. It is also direct to verify that $\displaystyle \sum_{j=1}^{n} (a_j-x_j)=0$ and $\lambda a_j+(1-\lambda)x_j=e+P_0(e)(\lambda a_j+(1-\lambda)x_j) $ for every $\lambda \in [0,1]$.\smallskip

Fix now $c>0$ satisfying $0 < c\leq \frac{\varepsilon}{4} \min \{\lambda_i:i\in \{1,\ldots,n \}\}$ and define $\mu_j=\displaystyle \frac{c}{\lambda_j}$. It is clearly satisfied that \begin{equation}\label{eq mu} \max \{ \mu_j : j \in \{1,\ldots,n \}\} =\frac{c}{\min \{\lambda_j:j\in \{1,\ldots,n \}\}} \leq \frac{\varepsilon}{4}. \end{equation}

We set $\gamma=1-\|P_0(e)(\sum_{i=1}^{n} \lambda_i x_i)\|$, which is positive by Lemma \ref{l support of convex combinations}.\smallskip

We can associate to every positive $\delta$, the following positive number  $\varepsilon_1=\varepsilon_1(\delta)=(m+1) \left(2\delta +\frac{4\sqrt{2}  \sqrt{\delta^2+4\delta }}{ \sqrt{3} \gamma}\right)+2\sqrt{2} \left(2\delta +\frac{4\sqrt{2}  \sqrt{\delta^2+4\delta }}{ \sqrt{3} \gamma}\right)^{\frac{1}{2}}$, where $m$ is the rank of the tripotent $e$, which satisfies $\lim_{\delta \to 0} \varepsilon_1(\delta)=0$.\smallskip

Take $\delta>0$ satisfying $\delta<\frac{\gamma}{4}$,  \begin{equation}\label{eq 3} 8 \varepsilon_1+\delta< \min \{\lambda_j:j\in \{1,\ldots,n \}\} (\frac{4-d^2}{4})
 \end{equation} and \begin{equation}\label{eq 4} 18\varepsilon_1 + 2\delta < \frac{\varepsilon}{2}.
 \end{equation}

Applying Remark \ref{r perturbation of supports} to any $y$ in the closed unit ball of $\mathcal{U}$ with $\|\sum_{i=1}^{n} \lambda_ix_i-y\|\leq\delta$ and denoting by $f$ the spectral resolution of $y$ associated to the set $[1-\delta,1]$, we have that \begin{equation}\label{eq 5}  \|f-e\|\leq \varepsilon_1.
 \end{equation}

We define next the elements $\tilde{x}_j$ and check the desired statements.\smallskip

For each $j\in \{1,\ldots, n\}$, we define $$\tilde{x}_j = P_2(f)y+P_0(f)[x_j+\mu_j(a_j-x_j)+y-\sum_{i=1}^{n} \lambda_i  x_i].$$ It follows straightforwardly that, $$ \sum_{j=1}^{n}  \lambda_j \tilde{x}_j = \sum_{j=1}^{n} \lambda_j P_2(f)y+P_0(f)[\sum_{j=1}^{n} \lambda_j x_j+\sum_{j=1}^{n} \lambda_j\mu_j(a_j-x_j) + \sum_{j=1}^{n} \lambda_j y- $$ $$\sum_{j=1}^{n} \lambda_j \sum_{i=1}^{n}\lambda_i  x_i]= P_2(f)y+P_0(f)[\sum_{j=1}^{n} \lambda_j x_j + c \sum_{j=1}^{n} (a_j-x_j) +  y-  \sum_{i=1}^{n}\lambda_i  x_i]=$$ $$P_2(f) y+P_0(f)y=y.$$
It is also satisfied that $\|x_j-\tilde{x}_j\|\leq \varepsilon$ for every $j\in \{1,\ldots,n\}$. Indeed, remembering that $\|\sum_{i=1}^{n} \lambda_ix_i-y\|<\delta$, we have $$\|x_j-\tilde{x}_j\|=\|x_j-P_2(f)y-P_0(f)[x_j+\mu_j(a_j-x_j)+y-\sum_{i=1}^{n} \lambda_i  x_i]\|= $$ $$\|P_2(f)x_j+P_1(f)x_j+P_0(f)x_j+f-f-P_2(f)y- P_0(f)x_j - \mu_jP_0(f)a_j+$$ $$\mu_j P_0(f)x_j -  P_0(f)(y-\sum_{i=1}^{n} \lambda_i  x_i)\|\leq \|P_2(f)x_j-f \| + \|f-P_2(f)y \| + \|P_1(f)x_j \| + $$ $$\mu_j\|P_0(f)(x_j-a_j) \|+ \|P_0(f)((y-\sum_{i=1}^{n} \lambda_i  x_i) \|\leq $$ $$ (\hbox{by (\ref{eq 5}), Lemma \ref{l continuity of peirce projections} and the definition of $f$ })\leq $$ $$9\varepsilon_1 + \delta + 9\varepsilon_1 + 2\mu_j+\delta \leq (\hbox{by (\ref{eq mu}) and (\ref{eq 4}) }) \leq \varepsilon.$$

Finally we will show that $\|\tilde{x}_j\|\leq 1$ for every $j\in \{1,\ldots,n\}$. Since $\|\tilde{x}_j\|=\max \{\|P_2(f)y\|,\|P_0(f)[x_j+\mu_j(a_j-x_j)+y-\sum_{i=1}^{n} \lambda_i  x_i] \| \},$ we only have to check that the second term is less than  or equal to 1. Now  $$\|P_0(f)[x_j+\mu_j(a_j-x_j)+y-\sum_{i=1}^{n} \lambda_i  x_i] \|\leq \|P_0(f)[(1-\mu_j)x_j+\mu_j a_j]\|+\|y-\sum_{i=1}^{n} \lambda_i  x_i\|\leq $$ $$\|P_0(e)[(1-\mu_j)x_j+\mu_j a_j]\|+\|(P_0(e)-P_0(f))[(1-\mu_j)x_j+\mu_j a_j]\|+\|y-\sum_{i=1}^{n} \lambda_i  x_i\| \leq $$ $$(\hbox{by Lemma \ref{l norm-control of convex combinations} and Lemma \ref{l continuity of peirce projections} \rm{a)}} ) \leq 1-\frac{4-d^2}{4}\mu_j+ 8\varepsilon_1+\delta \leq (\hbox{by (\ref{eq 3})}) \leq 1. $$

The case $\|\displaystyle \sum_{i=1}^{n} \lambda_i x_i\|<1$ is even simpler. Notice that in this case $e=0$ so that $P_2(e)=P_1(e)=0$ and $P_0(e)=\hbox{Id}_{|\mathcal{U}}$ and if  $\delta<\frac{\gamma}{4}$, the spectral resolution of $y$ corresponding  to the set $[1-\delta,1]$, $f$, is also zero. Defining $\tilde{x_j}$ in the same manner, with the less restrictive assumption $\delta < \min \{\frac{\gamma}{4},\frac{\varepsilon}{2}, \min \{ \frac{4-d^2}{4} \mu_j: j\in \{1,\ldots,n\} \}\}$ we arrive to the desired conclusion.\medskip

In order to prove that $\mathcal{U}$ has property $\hbox{\rm{(co)}}$ (see Definition \ref{d (co)}) and once we have fixed $\delta>0$, we only have to check that the functions $\phi_j:B(x,\delta)\cap \mathcal{B}_{\mathcal{U}}\to B(x_j,\varepsilon)\cap \mathcal{B}_{\mathcal{U}}$ defined by $\phi_j(y)=\tilde{x}_j$ for every $j\in \{1,\ldots,n\}$ are continuous.\smallskip

Given $y,z$ in $B(x,\delta)\cap \mathcal{B}_{\mathcal{U}}$, we have that $$ \|\phi_j(y)-\phi_j(z)\|=\| P_2(f_y)y+P_0(f_y)[x_j+\mu_j(a_j-x_j)+y-\sum_{i=1}^{n} \lambda_i  x_i]- $$ $$ P_2(f_z)z-P_0(f_z)[x_j+\mu_j(a_j-x_j)+z-\sum_{i=1}^{n} \lambda_i  x_i] \|=$$ $$ \|y-z+ (P_0(f_y)-P_0(f_z))[x_j+\mu_j(a_j-x_j)-\sum_{i=1}^{n} \lambda_i  x_i]\|\leq \|y-z\|+16\|f_y-f_z \|, $$ where in the last inequality we have used Lemma \ref{l continuity of peirce projections} and $\|x_j+\mu_j(a_j-x_j)-\sum_{i=1}^{n} \lambda_i  x_i \|\leq 2$. Theorem \ref{t continuity of spectral resolutions} assures that the functions $\phi_j$ are continuous.

\end{proof}

\end{document}